\documentclass[11pt,letterpaper]{amsart}

\def\arxivmode{1}
\def\fastmode{0}

\def\showauthornotes{1}

\def\showkeys{0}
\def\showdraftbox{1}
\def\showcolorlinks{1}

\def\showfixme{1}

\usepackage{xspace}
\usepackage{amsrefs}
\usepackage{amsfonts}
\usepackage{amssymb}

\usepackage[dvipsnames]{xcolor}

\ifnum\fastmode=0
\usepackage[T1]{fontenc}
\usepackage[full]{textcomp}
\fi

\usepackage[american]{babel}

\usepackage{mathtools}

\usepackage{amsthm}

\newtheorem{theorem}{Theorem}[section]
\newtheorem*{theorem*}{Theorem}

\newtheorem{proposition}[theorem]{Proposition}
\newtheorem*{proposition*}{Proposition}
\newtheorem{lemma}[theorem]{Lemma}
\newtheorem*{lemma*}{Lemma}

\newtheorem*{conjecture*}{Conjecture}

\newtheorem*{fact*}{Fact}

\newtheorem*{exercise*}{Exercise}

\newtheorem*{hypothesis*}{Hypothesis}

\theoremstyle{definition}
\newtheorem{definition}[theorem]{Definition}

\newtheorem*{question*}{Question}

\newtheorem{exercise-easy}[theorem]{Exercise}
\newtheorem{exercise-med}[theorem]{Exercise}
\newtheorem{exercise-hard}[theorem]{Exercise$^\star$}

\newtheorem*{claim*}{Claim}

\newtheorem{remark}[theorem]{Remark}
\newtheorem*{remark*}{Remark}

\newtheorem*{observation*}{Observation}

\usepackage[letterpaper,
top=1in,
bottom=1in,
left=1in,
right=1in]{geometry}

\ifnum\arxivmode=0
\usepackage{newpxtext} 
\usepackage{textcomp} 
\usepackage[varg,bigdelims]{newpxmath}
\usepackage[scr=rsfso]{mathalfa}
\usepackage{bm} 
\let\mathbb\varmathbb
\fi

\ifnum\arxivmode=1
\fi

\usepackage[scale=1.1]{inconsolata}

\ifnum\showkeys=1
\usepackage[color]{showkeys}
\fi

\makeatletter
\@ifpackageloaded{hyperref}
{}
{\usepackage{hyperref}}

\definecolor{bleudefrance}{rgb}{0.01, 0.1, 1.0}
\definecolor{azure}{rgb}{0.0, 0.5, 1.0}

\ifnum\showcolorlinks=1
\hypersetup{
colorlinks=true,
urlcolor=OliveGreen,
linkcolor=OliveGreen,
citecolor=OliveGreen}
\fi

\ifnum\showcolorlinks=0
\hypersetup{
colorlinks=false,
pdfborder={0 0 0}}
\fi

\usepackage{prettyref}

\newcommand{\savehyperref}[2]{\texorpdfstring{\hyperref[#1]{#2}}{#2}}

\newrefformat{eq}{\savehyperref{#1}{\textup{(\ref*{#1})}}}
\newrefformat{lem}{\savehyperref{#1}{Lemma~\ref*{#1}}}
\newrefformat{def}{\savehyperref{#1}{Definition~\ref*{#1}}}
\newrefformat{obs}{\savehyperref{#1}{Observation~\ref*{#1}}}
\newrefformat{ass}{\savehyperref{#1}{Assumption~\ref*{#1}}}
\newrefformat{thm}{\savehyperref{#1}{Theorem~\ref*{#1}}}
\newrefformat{cor}{\savehyperref{#1}{Corollary~\ref*{#1}}}
\newrefformat{cha}{\savehyperref{#1}{Chapter~\ref*{#1}}}
\newrefformat{sec}{\savehyperref{#1}{Section~\ref*{#1}}}
\newrefformat{app}{\savehyperref{#1}{Appendix~\ref*{#1}}}
\newrefformat{tab}{\savehyperref{#1}{Table~\ref*{#1}}}
\newrefformat{fig}{\savehyperref{#1}{Figure~\ref*{#1}}}
\newrefformat{hyp}{\savehyperref{#1}{Hypothesis~\ref*{#1}}}
\newrefformat{alg}{\savehyperref{#1}{Algorithm~\ref*{#1}}}
\newrefformat{rem}{\savehyperref{#1}{Remark~\ref*{#1}}}
\newrefformat{item}{\savehyperref{#1}{\ref*{#1}}}
\newrefformat{step}{\savehyperref{#1}{step~\ref*{#1}}}
\newrefformat{conj}{\savehyperref{#1}{Conjecture~\ref*{#1}}}
\newrefformat{fact}{\savehyperref{#1}{Fact~\ref*{#1}}}
\newrefformat{prop}{\savehyperref{#1}{Proposition~\ref*{#1}}}
\newrefformat{prob}{\savehyperref{#1}{Problem~\ref*{#1}}}
\newrefformat{claim}{\savehyperref{#1}{Claim~\ref*{#1}}}
\newrefformat{relax}{\savehyperref{#1}{Relaxation~\ref*{#1}}}
\newrefformat{red}{\savehyperref{#1}{Reduction~\ref*{#1}}}
\newrefformat{part}{\savehyperref{#1}{Part~\ref*{#1}}}
\newrefformat{ex}{\savehyperref{#1}{Exercise~\ref*{#1}}}
\newrefformat{property}{\savehyperref{#1}{Property~\ref*{#1}}}

\newcommand{\Sref}[1]{\hyperref[#1]{\S\ref*{#1}}}

\usepackage{nicefrac}

\ifnum\showauthornotes=2
\newcommand{\mynotes}[1]{{\sffamily\small\color{teal}{#1}}\medskip}
\newcommand{\Authornote}[2]{{\sffamily\small\color{Maroon}{[#1: #2]}}\medskip}
\newcommand{\Authornotecolored}[3]{{\sffamily\small\color{#1}{[#2: #3]}}}
\newcommand{\Authorcomment}[2]{{\sffamily\small\color{gray}{[#1: #2]}}}
\newcommand{\Authorstartcomment}[1]{\sffamily\small\color{gray}[#1: }

\newcommand{\Authorfnote}[2]{\footnote{\color{red}{#1: #2}}}
\newcommand{\Authorfixme}[1]{\Authornote{#1}{\textbf{??}}}
\newcommand{\Authormarginmark}[1]{\marginpar{\textcolor{red}{\fbox{\Large #1:!}}}}
\newcommand{\myexplain}[1]{{\sffamily\small\color{red}{\noindent [Explanation:\medskip\newline \begin{quote}#1\hfill]\end{quote}}}\medskip}
\newcommand{\explain}[1]{{\sffamily\small\color{red}{#1}}\medskip}

\else

\newcommand{\mynotes}[1]{}
\newcommand{\Authornote}[2]{}
\newcommand{\Authornotecolored}[3]{}
\newcommand{\Authorcomment}[2]{}
\newcommand{\Authorstartcomment}[1]{}

\newcommand{\Authorfnote}[2]{}
\newcommand{\Authorfixme}[1]{}
\newcommand{\Authormarginmark}[1]{}
\newcommand{\myexplain}[1]{}
\newcommand{\explain}[1]{}

\ifnum\showauthornotes=1
\renewcommand{\myexplain}[1]{{\sffamily\small\color{red}{\noindent \begin{quote}{\bf Explanation:} \medskip\newline #1\end{quote}}}\medskip}
\fi

\fi

\ifnum\showfixme=0

\fi

\usepackage{boxedminipage}

\newcommand{\abs}[1]{\lvert#1\rvert}
\newcommand{\Abs}[1]{\left\lvert#1\right\rvert}

\newcommand{\Norm}[1]{\left\lVert#1\right\rVert}

\newcommand{\E}{\mathbb{E}\xspace}
\renewcommand{\P}{\mathbb{P}\xspace}
\newcommand{\Var}{\mathbf{Var}\xspace}

\newcommand{\Ex}[1]{\E\left[#1 \right]}
\newcommand{\Px}[1]{\P\left(#1 \right)}

\usepackage{mathrsfs}

\newcommand{\textparen}[1]{\text{(#1)}}

\ifx\because\undefined
\newcommand{\because}[1]{\textparen{because #1}}
\else
\renewcommand{\because}[1]{\textparen{because #1}}
\fi

\newcommand\bdot\bullet

\ifx\mathds\undefined 

\else

\fi

\def\UD{\mathsf{Undesirable}}
\newcommand{\R}{\mathbb R}

\newcommand{\cE}{\Even}

\newcommand{\cH}{\mathcal H}

\newcommand{\cL}{\mathcal L}

\newcommand{\cN}{\mathcal N}
\newcommand{\cO}{\Odd}
\newcommand{\cP}{\mathcal P}
\newcommand{\cQ}{\mathcal Q}

\newcommand{\cZ}{\mathcal Z}

\renewcommand{\leq}{\leqslant}
\renewcommand{\le}{\leqslant}
\renewcommand{\geq}{\geqslant}
\renewcommand{\ge}{\geqslant}

\ifnum\showdraftbox=1

\else

\fi

\let\epsilon=\varepsilon

\numberwithin{equation}{section}

\newcommand{\addreferencesection}{
  \phantomsection
\ifnum\stocmode=0
  \addcontentsline{toc}{section}{References}
\else
  \addcontentsline{toc}{section}{References \hspace*{1in} --------- End of extended abstract ---------}
\fi

}

\newcommand{\sse}{\subseteq}

\newcommand{\eps}{\epsilon}

\allowdisplaybreaks

\usepackage{paralist}

\usepackage{comment}

\usepackage{etoolbox}

\def\SP{\mathsf{SingletonPolymers}}

\newcommand{\Q}{Q_{d}}
\newcommand{\Qp}{Q_{d, p}}

\newcommand{\Cnt}{\cZ_{d, p}}
\newcommand{\Cntl}{\cZ_{d, p, \lam}}
\newcommand{\Cnts}{\widehat{\cZ_{d, p}}}
\newcommand{\Cntsl}{\widehat{\cZ_{d, p, \lam}}}
\newcommand{\Good}{\mathsf{Good}}
\newcommand{\Bad}{\mathsf{Bad}}
\newcommand{\Col}{\mathsf{Collide}}
\newcommand{\Rep}{\mathsf{Repeat}}
\newcommand{\Nbr}{\mathsf{Neighbor}}
\newcommand{\N}{N_p}
\newcommand{\Win}{\mathsf{Window}}
\newcommand{\Even}{\mathsf{Even}}
\newcommand{\Odd}{\mathsf{Odd}}

\newcommand{\EM}{\mathsf{EvenMinority}}
\newcommand{\OM}{\mathsf{OddMinority}}

\newcommand{\ESmall}{\mathsf{Even}_{\mathsf{Small}}}
\newcommand{\OSmall}{\mathsf{Odd}_{\mathsf{Small}}}
\renewcommand{\c}{\mathsf{c}}
\newcommand{\TV}[1]{\Norm{#1}_{\mathrm{TV}}}

\newcommand{\Nice}{\mathsf{Ideal}}
\newcommand{\NiceE}{\Nice^\Even}
\newcommand{\NiceO}{\Nice^\Odd}

\newcommand{\AS}{\mathsf{ApproxSampler}}
\newcommand{\US}{\mathsf{UniformSampler}}
\def\ASE{\AS^{\Even, \Good}}
\def\DS{\mathsf{DefectSize}}

\newcommand{\Ep}{\E\Phi}

\renewcommand{\hat}[1]{\widehat{#1}}

\newcommand{\Fr}[2]{\left(\frac{#1}{#2}\right)}
\newcommand{\fr}[2]{\frac{#1}{#2}}
\newcommand{\f}[2]{\frac{#1}{#2}}

\renewcommand{\bar}[1]{\overline{#1}}

\newcommand{\df}{\coloneq}

\newcommand{\Exp}[1]{\exp\left(#1\right)}

\newcommand{\lam}{\lambda}

\renewcommand{\t}{\theta}
\newcommand{\sig}{\sigma}
\newcommand{\g}{\gamma}
\newcommand{\del}{\delta}

\newcommand{\om}{\omega}
\newcommand{\ze}{\zeta}

\newcommand{\Ber}{\mathrm{Bernoulli}}
\newcommand{\Bin}{\mathrm{Binomial}}
\newcommand{\Pois}{\mathrm{Poisson}}

\newcommand{\Nor}[2]{\mathrm{Normal}\left(#1, #2\right)}

\newcommand{\Cov}{\mathbf{Cov}}

\newcommand{\weakto}{\xlongrightarrow{\mathrm{d}}}
\newcommand{\pto}{\xlongrightarrow{\mathbb{P}}}

\renewcommand{\Box}[1]{\left[#1\right]}
\newcommand{\Rnd}[1]{\left(#1\right)}
\newcommand{\Bra}[1]{\left\{#1\right\}}

\newcommand{\mnorm}[1]{{\left\vert\kern-0.25ex\left\vert\kern-0.25ex\left\vert #1 
		\right\vert\kern-0.25ex\right\vert\kern-0.25ex\right\vert}}

\newcommand{\1}[1]{\mathbf{1}_{\Bra{#1}}}

\begin{document}

\title[Independent sets of a percolated hypercube]{Gaussian to Log-Normal transition for independent sets in a percolated hypercube}

\author{Mriganka Basu Roy Chowdhury, Shirshendu Ganguly and Vilas Winstein}
\address{Mriganka Basu Roy Chowdhury\\ University of California, Berkeley}
\email{mriganka\_brc@berkeley.edu}
\address{Shirshendu Ganguly\\ University of California, Berkeley}
\email{sganguly@berkeley.edu}
\address{Vilas Winstein\\ University of California, Berkeley}
\email{vilas@berkeley.edu}

\begin{abstract} Independent sets in graphs, i.e., subsets of vertices where no two
	are adjacent, have long been studied, particularly as a model of hard-core
	gas in various settings such as lattices and random graphs. 
	The $d$-dimensional hypercube, $\{0,1\}^d$, with the nearest neighbor
	structure, has been a particularly appealing choice for the base graph, owing
	in part to its many symmetries, as well as its connections to binary codes.
	Results go back to the work of Korshunov and Sapozhenko \cite{sapozhenko} 
	who proved sharp
	results on the count of such sets as well as structure theorems for random
	samples drawn uniformly. This was later extended by Galvin \cite{g11} 
	to the hard-core model, where the probability of an independent
	set $I$ is proportional to $\lambda^{|I|}$ where $\lambda>0$ is the
	\emph{fugacity}.

	Of much interest is the behavior of such Gibbs measures in the presence of
	disorder. In this direction, Kronenberg and Spinka \cite{ks} initiated the study of the count and
	geometry of independent sets in a random subgraph of the hypercube obtained
	by considering an instance of bond percolation with probability $p.$
	Relying on tools from statistical mechanics, such as polymer decompositions
	and cluster expansions, they obtained a detailed understanding of the
	moments of the partition function, say $\cZ,$ of the hard-core model on such
	random graphs and consequently deduced certain fluctuation information, as
	well as posed a series of interesting questions. A particularly important
	outcome of their computations is that, in the uniform case, there is a
	natural phase transition at $p=\f 23$ where $\cZ$ transitions from being
	concentrated for $p>\f 23$ to not concentrated at $p=\f 23$ in the sense that
	the second moment transitions from being much smaller than the square of the expectation
	to being comparable at $p=\f23$, to being exponentially larger for $p>\f23.$
 
	In this article, developing a probabilistic framework, as well
	as relying on certain cluster expansion inputs from \cite{ks}, we present a
	detailed picture of both the fluctuations of $\cZ$ as well as the geometry of
	a randomly sampled independent set. In particular, we establish that $\cZ,$
	properly centered and scaled, converges to a standard Gaussian for $p>\f 23,$
	and to a sum of two i.i.d. log-normals at $p=\f 23.$ 
A particular step in the proof which could be of independent interest involves a non-uniform birthday problem for which
collisions emerge at $p=\f 23,$ establishing this as another natural transition point.
	Finally, our results also generalize
	to the hard-core model with parameter $\lambda,$ establishing a similar
	transition at $\frac{(1+\lambda)^2}{2\lambda(2+\lambda)}.$   
\end{abstract}

\maketitle

\setcounter{tocdepth}{1}
\tableofcontents

\section{Introduction}
An independent set in a graph $G=(V,E),$ is a subset of the vertices of which no two are neighbors. A canonical model of
interest in statistical mechanics, they arise as a model of hard-core lattice gas or in the study of glassy properties of
packed hard-spheres as well as a toy model for the antiferromagnetic Ising model. As is common in statistical mechanics,
often a useful framework to study such objects is by considering a suitable Gibbs measure defined on them. This is known
as the \emph{hard-core model} indexed by the \emph{fugacity} $\lambda >0$. Given
$G$, the hard-core model assigns a probability on any independent set $I\sse V$, is 
\begin{equation}\label{hardcoredef342}
\frac{\lambda^{|I|}}{\cZ},
\end{equation}
where $|I|$ denotes the size of $I$ and $\cZ$, the normalizing constant, is termed as the partition function. A central
question for any such Gibbs measure is to compute this partition function which is rich in information about the
structure of the underlying measure. Note that when $\lambda=1$, the partition function is simply the count of the
number of independent sets that the graph admits. While the problem of independent sets can be studied for any graph, of
particular interest in both mathematics and computer science  are the cases when the underlying graph is a lattice or a
random graph. The hard-core models in both these cases are known to exhibit a rich structure and phase transitions (often
multiple) as the fugacity $\lambda$  is varied. On the lattice, owing to its bipartite structure, as $\lambda$ increases
to infinity, the geometry of the independent set adopts a more checkerboard appearance approaching its ground state, see
for instance, \cite{g04}. The random graph setting falls in the paradigm of random constraint satisfaction problems (the
random instance given by the underlying graph). A remarkable phase transition in the geometry of the set of independent
sets has been established over the years exhibiting various phenomena ranging from the replica symmetric regime where
the Gibbs measure behaves as a product measure, to  condensation and shattering as the density of the independent set is
varied, where the set of all independent sets breaks up into various clusters. See, for instance, 
\cites{dingslysunregular, c15, b13} for state of the art results in this direction. The powerful hypergraph container method was introduced in the seminal works \cite{hypergraph_containers, balogh2015independent} to study independent sets in hypergraphs.

The particular case treated in this paper is a combination of the above two setups. The base graph considered is the
high dimensional lattice given by the hypercube $\Q\df\{0,1\}^d$, endowed with the nearest-neighbor graph structure,
where two vertices $u$ and $v$ are neighbors if and only if they differ at exactly one bit. Owing to the various
symmetries of $\Q$, the independent set problem in this setting has received quite a lot of attention over the years.
The articles \cites{j22,g11, perkins} and the references therein should provide a comprehensive account. It is also
related to the problem of binary codes at distance two, see for instance \cite{sapozhenko} making it of wider
interest.  As a consequence, the current understanding of the count of independents sets as well as their typical
structure is quite advanced which we review briefly next.

Results go back to the beautiful work of Korshunov and
Sapozhenko \cite{sapozhenko} who computed the leading behavior of the count of
independent sets, say $\cZ_{d}$, showing that as $d \to \infty$,
\begin{equation}\label{p1count}
    \cZ_{d} = 2^{2^{d - 1}} \cdot 2 \cdot \sqrt{e} \cdot (1 + o(1)),
\end{equation}
see also \cite{galvinind} for a more recent exposition of the proof. This was greatly generalized by Jenssen and Perkins
\cite{perkins}, who employed tools from statistical mechanics such as cluster expansions and \emph{theory of abstract polymer models}, extracting finer corrections to this count for arbitrary {fugacities} $\lam$ in the corresponding hard-core model.
For structural results, note that $\Q$ admits a natural
bipartition $\Q = \cE \cup \cO$ depending on whether the vertex has even or odd Hamming weight, respectively.
Kahn \cite{kahn}
showed that for constant $\lambda$, typical independent sets drawn from $\mu$ contain either mostly
even vertices or mostly odd vertices, and thus the hard-core model on $\Q$ exhibits a
kind of ‘phase coexistence’ as alluded to in the beginning of this discussion. 
In particular, for a uniformly chosen independent set, corresponding to $\lambda=1$, most vertices land on the
\emph{majority} side with the opposite side featuring as the \emph{minority} side.

While the story of the study of independent subsets of $\Q$ is significantly richer than the scope of the quick review
we have attempted here, this paper focuses on how such Gibbs measures behave in the presence of disorder or impurities.

This general program of studying the effect of disorder on Gibbs measures can be traced back to the study of spin glasses in the context of the Ising model (see \cite{panchenko2013sherrington} for instance).  For the independent set problem on $\Q$, a
natural way to introduce disorder is to alter the underlying geometry by considering an instance of bond percolation
where every edge is retained independently with probability $p.$ We will denote the random subgraph obtained this way as
$\Qp.$ The study of the independent set problem on $\Qp$ was initiated in the recent work of Kronenberg and Spinka
\cite{ks}. The primary object of interest then becomes $\Cnt$, the number of independent sets that $\Qp$ admits. {
	To understand $\Cnt$, a quenched quantity which is a function of the percolation instance, \cite{ks} considers the annealed model, i.e., obtained by averaging over over the percolation configuration. An analysis of the annealed model is then performed using ideas from statistical mechanics related to the ones appearing in the aforementioned article \cite{perkins}. The observation that the annealed model admits tractable cluster expansions allows them to compute moments of $\Cnt$ via computing the partition functions
	for increasingly complex polymer models.
	Relying on this, they prove the following result, analogous to \eqref{p1count}:
}

\begin{align}\label{expectation12}
    \E \Cnt = 2^{2^{d - 1}} \cdot 2 \cdot \Exp{\f{1}{2} (2 - p)^d \cdot (1 + o(1))}.
\end{align}
Further, their control on the moments also allow them to deduce certain fluctuation statements such as the following.
\begin{theorem*}\cite{ks}*{Theorem 1.2} For any fixed $p \geq \f 23$,
    $$\Cnt= 2^{2^{d - 1}} \cdot 2 \cdot \Exp{\f{1}{2} (2 - p)^d}\cdot\left(1+\Rnd{2-\frac{3p}{2}}^{d/2}X_d \right),$$
where $X_d$ is tight as $d\to \infty.$  Further, $X_d,$ does not converge to zero. 
\end{theorem*}
While \cite{ks} also deals with $p$ depending on $d,$ we will simply quote their results for the fixed $p$ case.  \\

\begin{figure}[ht]
\centering
\includegraphics[scale=.08]{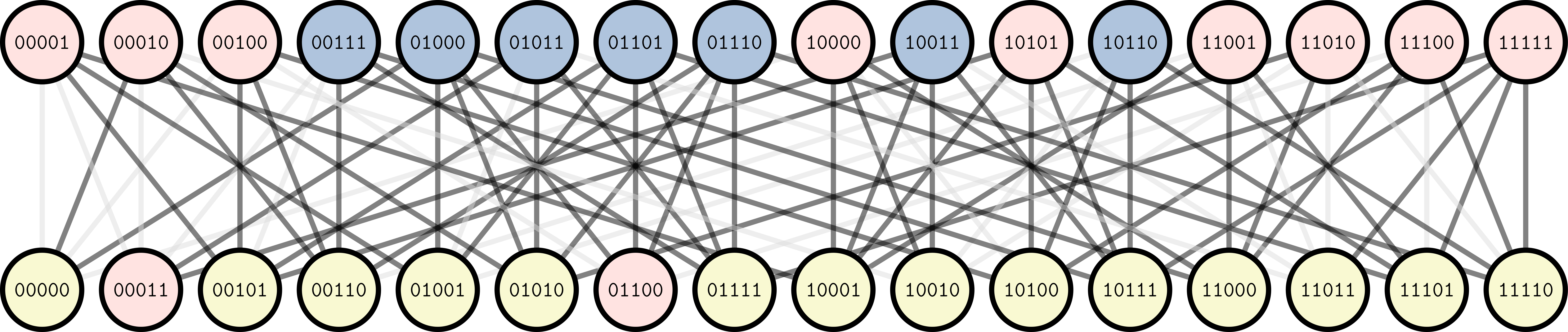}
\caption{An illustration of $\Qp$ (open edges in black) with $\Even$ (yellow) and $\Odd$ (blue) sides. Vertices in pink form an independent
	set.
}
\label{illus1234}
\end{figure}

Note that the above implies a natural phase transition at $p=\f 23,$ where for
$p> \f 23$, since $(2-\f{3p}{2})<1,$ we see that $\Cnt$ is concentrated, 
while for $p=\f23$ it is not. Further, they prove that the variance
transitions from being much smaller than the square of the expectation for
$p>\f23$ (also implying concentration), to being comparable at $p=\f23$, to being
much larger than, for $p<\f23.$ 

While it turns out that moment estimates of $\Cnt,$ in general, will not be particularly useful to capture fluctuations or
typical behavior, they indeed suffice to yield surprisingly sharp results when $p$ is particularly close to $1$ \emph{as a function
of $d$}. In this vein, they had established a Gaussian fluctuation result of the following kind.

\begin{theorem*}\cite{ks}*{Theorem 1.3}\label{clt1} For $p=1-o(1)$ such that $p\le 1-2^{-d/3+\omega(1)},$
$$\frac{\Cnt-\E\Cnt}{\sqrt{\Var(\Cnt)}}\weakto \Nor{0}{1},$$
where $\weakto$ denotes convergence in distribution.
\end{theorem*}

Their primary method relies on the observation that the moments of $\Cnt$ correspond to the partition function of a
related \emph{positive temperature} model allowing them to employ techniques such as polymer representation and cluster
expansion, see \cite{velenik}*{Chapter 5}, some of which will also serve as key inputs in our arguments. Finally, given the above findings,
several natural questions arise, some of which were already recorded in \cite{ks}.  We list a few of them pertinent to this article, below.\\

\noindent
$1.$ What is the fluctuation theory of $\Cnt$ for various values of $p?$ \\
$2.$ For what values of $p$ does a CLT analogous to the above theorem hold? \\
$3.$ What happens at $p=\f23?$ \\
$4.$ How does a uniformly chosen independent set in $\Qp$ look?\\
$5.$
 What about the general hard-core model?\\

\emph{In this article, we adopt a more probabilistic perspective to develop a comprehensive understanding of the count and geometry of 
independent sets in $\Qp$, and in the process answering all the above questions.}\\

First, we introduce two quantities $\mu_p \coloneq \mu_{d, p} \coloneq \fr{1}{2}(2 - p)^d$ and
$\sig_p^2 \coloneq \sig^2_{d, p} \coloneq \fr{1}{2}\Fr{4-3p}{2}^d$ that will feature
prominently in our arguments. Their centrality will be apparent soon, but
momentarily, without any further elaboration on their relevance, we proceed to
the statements of our results.\\

Define the scaled count
\begin{align}
    \label{eq:scaledcnt}
    \Cnts \df \fr{\Cnt}{2\cdot 2^{2^{d - 1}} \exp(\mu_p)}.
\end{align}

Then, our first main result is the following.

\begin{theorem}
    \label{thm:main}
  For $p \in (\f23,1)$, 
    \[
        \fr{\Cnts - 1}{2^{-{\f 12}} \cdot \sigma_{p}} \weakto \Nor{0}{1}.
    \]
    Whereas, for $p = \f23$, the limiting distribution is a sum of log-normals, i.e.,
    \[
        \Cnts \weakto \fr{e^{-{\f 14}}}{2^{\f 32}}
            \Rnd{\Exp{\fr{W_1}{\sqrt{2}}} + \Exp{\fr{W_2}{\sqrt{2}}}}
    \]
    where $W_1, W_2 \overset{\mathrm{i.i.d.}}{\sim} \Nor{0}{1}$.
\end{theorem}

Thus, the above theorem answers the first three questions above for all $p\in[\f23 ,1).$ Note that we have excluded the case $p=1$. This is deliberate since this has already been addressed in several previous works. Thus, refraining from including this case in our results, we encourage the interested reader to review some of the relevant articles mentioned in the introduction and the references therein. Moreover, while as in \cite{ks} one may
also consider $p$ depending on $d,$ we have not pursued this to ensure that the arguments remain the most concise.  
Our
arguments also allow us to generalize the above result to the case of a general hard-core model defined in \eqref{hardcoredef342} with a fixed fugacity $\lambda$ as long as it is greater than $\sqrt 2-1.$ The source of this constraint will be made clear shortly once we present the counterpart results. In this
setting, the exact same statement holds with \begin{align}\label{modconstant123}
\mu_{p, \lam}\coloneq    \mu_{d,  p, \lam} &\df \fr{\lam}{2} \Rnd{2 - \fr{2\lam p}{1 + \lam}}^d,\\
    \nonumber
  \sigma_{p,\lam}^2 \coloneq  \sig_{d, p \lam}^2 &\df \fr{\lam^2}{2}\Rnd{2\Box{1 - p + \fr{p}{(1 + \lam)^2}}}^d,
\end{align}
and, 
\begin{align}
    \label{eq:scaledcntlambda}
    {\Cntsl} \df \fr{\Cntl}{2\cdot (1+\lambda)^{2^{d - 1}}\exp{(\mu_{p,\lambda})}
}. \end{align}

The proof of Theorem \ref{thm:main}, essentially as is, implies that in this setting, for $p > \fr{(1 + \lam)^2}{2\lam (2
+ \lam)}$ (note that as expected this equals $\f 23$ at $\lambda=1$),
\begin{align}\label{hard12}
		\fr{\Cntsl - 1}{2^{-\f12} \cdot \sig_{p, \lam}} \weakto \Nor{0}{1},
\end{align}
    whereas, for $p = \fr{(1 + \lam)^2}{2\lam (2 + \lam)}$, the limiting distribution is a sum of log-normals, i.e.,
\begin{align}\label{hard13}
	\Cntsl \weakto \fr{e^{-\fr{\lam^2}{4}}}{2^{\f32}}
            \Rnd{\Exp{\fr{\lam W_1}{\sqrt{2}}} + \Exp{\fr{\lam W_2}{\sqrt{2}}}}
\end{align}
	where $W_1, W_2 \overset{\text{i.i.d.}}{\sim} \Nor{0}{1}$. Observe that the constraint $\lambda > \sqrt{2}-1$ arises
	naturally from the requirement that the location of the transition $\fr{(1 + \lam)^2}{2\lam (2+ \lam)}$ lies in
	$[0,1).$ It is worth remarking that the threshold $\sqrt 2 - 1$ has made an appearance in earlier works as well; see
	for instance \cite{g11}. 

Our next result addresses the structure of an independent sample uniformly, or more generally from the hard-core model.
Note that some care is needed to state the results, since, as is common for random Gibbs measures, there are two layers
of randomness, one from the underlying percolation, and then given the latter, the associated Gibbs measure. Often a
more desirable statement is a quenched one, i.e., one conditioned on the percolation environment and this is what we
establish next. As above, we will state the result in the special case of $\lambda=1$ for notational simplicity, but again
the result extends to the general case which we will comment on in Section \ref{hardcore3456}. 

\begin{theorem}
    \label{thm:geometry}  Assume $p\ge \f23$ and let  $\US$ be the uniform distribution on independent sets under a quenched realization of the percolation, and let
    $\DS$ be the distribution of $\min(|S \cap \Even|, |S \cap \Odd|)$ for $S \sim \US$. Then, 
    \[
        \TV{\DS - \Pois(\mu_p)} \pto 0,
    \]
    where $\TV{\cdot - \cdot}$ is the total-variation distance. Note that $\pto 0$ implies that with high probability the percolation configuration is such that $\TV{\DS - \Pois(\mu_p)} $ is close to zero.
\end{theorem}

Using the well known fact that a Poisson random variable with a large mean (since it is a sum of independent Poissons
with mean $1$ plus some error) is close to a Gaussian, as an immediate corollary, we get a CLT for the size of the
defect side.

A key input in the proof of the above result is the following sampling procedure
producing an approximately uniform independent set, which could be of broader interest.
To state the algorithm, let $\N(S)$ denote the number of neighbors in $\Qp$
of a set $S$. Sometimes to avoid introducing further notation, we will also use $\N(S)$ to denote  the {\it set} of neighbors. The context of the usage will be apparent and will not present a scope for confusion. For the
 full hypercube, we will drop the $p=1$ subscript and simply use $N(S)$. Finally, for a single vertex, we use $\N(v)$ for $\N(\{v\})$.

\begin{definition}[Approximate sampler]
    \label{def:sampling}
    Given a realization of the percolation $\Qp$, we sample an independent set $S \sim \AS$ as follows:
\begin{enumerate}
    \item Choose a side $\cH = \Even$ or $\Odd$ with probability ${\f 12}$.

	\item For each vertex $v \in \cH$, pick it independently with probability $\fr{2^{-N_p(v)}}{1 + 2^{-N_p(v)}}$, and collect these vertices in the set $S_1$.
    \item For each vertex $u$ in the other side $\cH^\c$ not in $N_p(S_1)$, i.e., for every $u \in \cH^\c \setminus
        N_p(S_1)$, pick it independently with probability ${\f 12}$. Collect these vertices in the set $S_2$.
    \item Output $S = S_1 \cup S_2$.
\end{enumerate}
\end{definition}

Our final result then states that the output of the aforementioned sampling procedure is not very
different from a uniformly chosen independent set in the case $p>\f23$. 

\begin{proposition}[Properties of $\AS$]
    \label{lem:sampling1} For $p>\f23$, the total-variation distance between $\AS$ and $\US$ vanishes asymptotically in probability,
    i.e.,
    \[
        \TV{\AS - \US} \pto 0.
    \]
\end{proposition}
The case $p=\f23$ is a bit different and the sides $\cE$ and $\cO$ are asymmetrically biased and hence $\cH$ will need to be sampled according to a random distribution. The precise statement in this case needs a few further definitions and is deferred to Section \ref{sec:geometry} where it will be proved along with the $p>\f23$ case.\\

Having stated our main results, we now move on to a discussion of the important ideas in our arguments as well lay down
the key intermediate results we prove en route. In this section we also finish the proof of Theorem \ref{thm:main}
modulo these intermediate results, whose proofs occupy the rest of the paper.

\section{Idea and proof of main theorem}
\label{sec:idea}

\subsection{Overview}
As already mentioned, the hypercube is bipartite with the parts denoted by $\Even$ or $\Odd$. Consequently every
subset of $\Q$ has an even side and an odd side, obtained by the intersections with $\Even$ and $\Odd$ respectively.
Further, recall that for any $v \in \Q$, $\N(v)$  is the (random) number of neighbors
of $v$ in $\Qp$ and 
similarly, for a subset $S \sse \Q$, we define $\N(S)$ to be the number of 
vertices which neighbor any vertex of $S$ in $\Qp$.
This will exclusively be used for sets $S$ which are contained in either the 
even or the odd side of $\Q$.

Now, for any $S \sse \Even$, the number of independent sets with even side $S$
is exactly
\begin{align*}
    2^{2^{d-1} - \N(S)},
\end{align*}
since there are $2^{d-1} - \N(S)$ vertices in the odd side which are not 
excluded from being in an independent set by the inclusion of $S$, and we can
take an arbitrary subset of them to obtain an independent set.
A similar formula holds for the odd side with $S \sse \Odd$.

The key phenomenon that drives much of the argument in the paper is that most independent sets have a clear majority and
a minority side, i.e, either the even side is much bigger than the odd side or vice versa. We will often term the
minority side as the \emph{defect} side following previous work. 

Most of the analysis in the paper is then devoted to analyzing the contributions of the various defect side possibilities.
One possibility is that the defect side solely consists of singletons that are well separated. Another possibility is
that there are elements on the defect side which are $2-$neighbors of each other and form components of size larger than
one,  which we term in the paper as being
$2-$linked. 

In the most basic case where the defect side has cardinality two, a quick back of the
envelope computation reveals that the \emph{expected} number of independent sets with two well separated points is much
larger than those with two points which are $2-$linked. However, the various possibilities of defect sets quickly grow
with its size and an efficient systematic way to deal with them is using cluster expansion where $2-$linked  components
are termed as polymers. This is what the main contribution of \cite{ks} is. However, the inputs from \cite{ks} only allow
access to expected quantities. The fact that we rely on at this point is that $\sigma^2_p=\fr{1}{2}\Fr{4-3p}{2}^d$ is
exponentially small for $p>\f 23$ suggesting strong concentration of $\Cnts$ allowing us to replace the random $\Cnts$ by
its expectation. Thus, we seek to show that the expected number of independent sets with any one of the following
undesirable properties is negligible compared to $\E\Cnts$, following which a simple application of Markov's inequality
suffices:\\
 
 \noindent
$1.$ The sizes of the two sides are not too different.\\
\noindent
$2.$ The size of the minority side itself is not too small in a sense made precise in Section \ref{sec:polymer}.\\
\noindent
$3.$ The minority side has $2-$linked components of size bigger than one.\\

The concentration phenomenon fails to hold for $p\le \f23.$ While the critical case $p=\f23$ can be analyzed using the
arguments of this paper, new arguments are needed when $p<\f23$. This will be the subject of a forthcoming work. We remark
further on this later in the article (see Remark \ref{below23}).

The above discussion indicates that presently we can restrict our analysis to sets where the majority side is rather
large, and the minority or the defect
side  is small and well-separated, meaning that no two 
vertices share a common neighbor on the opposite side.
We denote by $\Good^\Even$ the collection of such possible defect sets on the even side and similarly $\Good^\Odd$ (the
class of all such sets will sometimes be termed simply as $\Good$ for brevity.)
Thus putting the above together, the following approximate equality holds.
\begin{align}\label{polyinput}
    \Cnt \approx 2^{2^{d-1}} \Rnd{\sum_{S \in \Good^\Even} 2^{-\N(S)} + \sum_{S \in \Good^\Odd} 2^{-\N(S)}}.
\end{align}
Notice that for $S \in \Good$, the fact that $S$ is well-separated
implies that $\N(S) = \sum_{v \in S} \N(v)$.
{This important feature allows us to factor the summands in the above sum and
obtain a more tractable expression for $\Cnt$.}
Indeed, we have
\begin{align}\label{goodset32}
    \sum_{S \in \Good^\Even} 2^{-\N(S)} 
    = \sum_{S \in \Good^\Even} \prod_{v \in S} 2^{-\N(v)}.
\end{align}
In the remainder of this section as well as throughout the rest of article, it will be convenient to let, for any $v \in
\Q$, \begin{equation}\label{brevity12}
\varphi_v := 2^{- \N(v)}.
\end{equation}
Note then that, in \eqref{goodset32}, the sum over all $S$ with $|S| = m$ also 
appears in the expansion of
\begin{align}
    \label{eq:expsummand}
    \fr{1}{m!} \Rnd{\sum_{v \in \Even} \varphi_v}^m,
\end{align}
which tempts us to approximate the quantity in \eqref{goodset32} by
\begin{align}\label{expapproximation12}
        \Exp{\sum_{v \in \Even} \varphi_v}.
\end{align}
However, for such a strategy to work, the remaining terms in the expansion of $\Rnd{\sum_{v \in \Even} \varphi_v}^m$, i.e.,
of the form $\varphi_{v_1}\varphi_{v_2}\ldots \varphi_{v_m}$ where the $v_{i}$s are not separated (this includes both when $v_i$
is equal to or a $2-$neighbor of $v_j$ for some $i\neq j$), must form a negligible fraction. 

At this point, it is convenient to recast this in terms of the following non-uniform birthday problem: the fraction
which we are aiming to show is negligible is exactly the collision probability (where a collision refers to $v_i$ being
equal to or a $2-$neighbor of $v_j$ for some $i\neq j$) {when $m$ samples are drawn independently with probability of $v$
proportional to $\varphi_v.$} Perhaps somewhat unexpectedly at first glance, $p=\f 23$ also turns out to be the point where
collisions start becoming likely and hence for $p> \f 23,$ this shows that indeed the collision terms can be ignored. 

Nonetheless, recall that Theorem \ref{thm:main} also covers the $p=\f23$ case. In this case
the undesirable terms account for a nontrivial
proportion of \eqref{eq:expsummand}.
This is addressed relying on an application of the Stein-Chen method. A wonderful survey can be found in \cite{ross}. Using the latter, we show that the number
of collisions resembles a Poisson distribution whose mean does not change much as the number of samples $m$ is varied in
a suitable window of interest. This allows us to calculate the limiting proportion of desirable terms (the probability
that the corresponding Poisson variable is zero), and we find that in the $p=\fr{2}{3}$ case this is $e^{-{\f 14}}$
which leads to replacing \eqref{eq:expsummand} by $e^{-{\f 14}}\fr{1}{m!} \Rnd{\sum_{v \in \Even} \varphi_v}^m$ and hence
the approximation
\begin{align*}
    \sum_{S \in \Good^\Even} 2^{-\N(S)} \approx
        e^{-{\f 14}} \Exp{\sum_{v \in \Even} \varphi_v}.
\end{align*}
Of course, all of the above analogous approximations in both cases also hold if we replace
$\Even$ by $\Odd$.
Putting this all together, we have the approximation
\begin{align}
    \label{eq:mainapprox}
    \fr{\Cnt}{2^{2^{d-1}}} \approx 
    \zeta_p \Rnd{\exp(\Phi_\Even) + \exp(\Phi_\Odd)},
\end{align}
where we have defined $\zeta_p = 1$ for $p > \fr{2}{3}$ and 
$\zeta_{\f 23} = e^{-{\f 14}}$, and 
\begin{align}
    \label{eq:psidef}
    \Phi_\Even \coloneqq \sum_{v \in \Even} \varphi_v
    \qquad \text{and} \qquad
    \Phi_\Odd \coloneqq \sum_{v \in \Odd} \varphi_v. 
\end{align}
Now it turns out that $\E(\Phi_{\cE})=\mu_p$ and $\Var(\Phi_{\Even})=(1+o(1))\sigma_p^2$ which explains their appearance
in Theorem \ref{thm:main} (see Lemma \ref{lem:exvarcov} later where this is worked out). Thus, writing 
\begin{align}\label{iidsum}
\Phi_{\cE}=\mu_p+\sigma_p W_{\cE},\,\,\text{and}\,\,\,\, \Phi_{\cO}=\mu_p+\sigma_p W_{\cO},
\end{align}
we get 
\begin{align}
\fr{\Cnt}{2^{2^{d-1}}} &\approx 
    \zeta_p \Rnd{\exp(\Phi_\Even) + \exp(\Phi_\Odd)}\\
    &\approx \zeta_p \exp(\mu_p)\left[\exp(\sigma_pW_{\cE})+\exp(\sigma_p W_{\cO})\right].
\end{align}
At this point note that when $p>\f23,$ $\sigma_p$ decays to zero exponentially fast while for $p=\f 23$, $\sigma_p=\f1{\sqrt 2}.$
Hence, in the former case, we further have the approximation  
\begin{equation}\label{linear}
\exp(\sigma_pW_{\cE})+\exp(\sigma_p W_{\cO})\approx 2+\sigma_{p}(W_{\cE}+W_{\cO}).
\end{equation}

Given both $\Phi_{\cE}$ and $\Phi_{\cO}$ are sums of independent random variables $\varphi_v$ with $v\in \cE$ and $v\in \cO$ the plausibility of Theorem \ref{thm:main} is now as apparent as that of a central limit theorem for $\Phi_{\cE}$ and $\Phi_{\cO}.$ However, a potential issue is posed by the fact that while marginally a sum of i.i.d variables, $\Phi_{\cE}$ and $\Phi_{\cE}$ share the same randomness and hence are not independent of each other. Nonetheless, the dependency can be shown to be weak enough. While multiple approaches may work, we rely again on an application of Stein's method to prove a joint CLT for the pair $(\Phi_\Even, \Phi_\Odd)$ which suffices for our deductions (this is recorded as Lemma 
\ref{lem:iop_jointclt} below). \\

We conclude this discussion with a brief sketch of the argument leading to Theorem \ref{thm:geometry}. \\

\noindent
\textbf{Structure of independent sets.} It follows, say from \eqref{eq:expsummand}, that the fraction of independent
sets  whose minority side is $\cE$ and has size $m$ is approximately proportional to $\fr{\Phi^m_{\Even}}{m!}$ (and
similarly for $\cO$). Thus, conditioned on $\Even$ being the minority side, the size of the same is approximately
distributed as a Poisson variable with mean $\Phi_{\Even}.$ The concentration of the latter around its mean $\mu_p$
allows us to replace $\Pois(\Phi_{\cE})$ by $\Pois(\mu_p).$ Our description also leads to a natural algorithm to sample
independent sets which we don't elaborate further here in the interest of brevity. \\

\begin{remark}\label{below23} We end with a brief commentary on what one may expect for $p<\f23$. Two notable
	distinctions occur. First, $\sigma_p$ grows exponentially. Second, the birthday problem starts seeing exponentially
	many collisions. Nonetheless, the cluster expansion estimates still indicate that polymers of size $2$ or more
	continue to be rare all the way down to $p > 2-\sqrt 2.$ Thus, \eqref{eq:mainapprox}, which no longer stays valid,
	must be replaced by suitable pre-factors  to take into account the non-trivial effect of collisions. For instance, the approximation $(1+x)\approx e^x$ for $x$ small which turns out to be enough for this paper must be replaced by an approximation of the form $(1+x)\approx e^{x-\frac{x^2}{2}}.$ This general approach also allow us to treat the more interesting regime beyond $2-\sqrt 2$ where dimers (polymers of size two) appear.
	All of this will be the subject of a forthcoming work \cite{bgw2}.
\end{remark}

The remainder of this section outlines the key results we prove towards implementing the above strategy, and then
finishes the proof of Theorem \ref{thm:main} in Section \ref{sec:mainproof} assuming those. The consequence of the joint
CLT result  we require appears as Lemma \ref{lem:iop_jointclt} in Section \ref{jointCLT}.   The cluster expansion input
delivers Lemma \ref{lem:iop_polymer} which justifies \eqref{polyinput}. Lemma \ref{lem:iop_birthday} addresses the
birthday problem and states what the collision probability is.  The rest of the paper then is devoted towards
establishing these lemmas (the structure of the paper is recorded in Section \ref{org34}).

\subsection{Joint CLT for $(\Phi_\Even, \Phi_\Odd)$}\label{jointCLT}
While we will prove the following in Section
\ref{sec:jointclt}, 
\begin{align}\label{jointclt12}
    \Rnd{\fr{\Phi_\Even - \mu_p}{\sigma_p},\fr{\Phi_\Odd - \mu_p}{\sigma_p}}
    \weakto \Nor{0}{1} \otimes \Nor{0}{1},
\end{align}
we record here the straightforward consequence that we will need
for Theorem \ref{thm:main}.
\begin{lemma}
\label{lem:iop_jointclt}
If $p>\fr{2}{3}$, then 
\begin{align*}
    \fr{\exp(\Phi_\Even) + \exp(\Phi_\Odd) - 2 \exp(\mu_p)}
    {\sqrt{2} \cdot \sigma_p \exp(\mu_p)}
    \weakto \Nor{0}{1}.
\end{align*}
If $p=\fr{2}{3}$, then 
\begin{align*}
    \fr{\exp(\Phi_\Even) + \exp(\Phi_\Odd)}{\exp(\mu_p)} 
    \weakto \Exp{\fr{W_1}{\sqrt{2}}} + \Exp{\fr{W_2}{\sqrt{2}}},
\end{align*}
where $W_1$ and $W_2$ are i.i.d.\ $\Nor{0}{1}$.
\end{lemma}

\subsection{Reduction to $\Good$ defect side via cluster expansion results}
Here we state the bound, which says that for the purposes of proving Theorem \ref{thm:main} as well as Theorem \ref{thm:geometry}, it will suffice to simply consider independent sets whose one side is smaller than $\frac{2^{d}}{d^2}$ as well as separated without any $2-$neighbors. These defect sets will be termed $\Good^\Even$ and $\Good^\Odd$
respectively depending on them being on the even or the odd side. A formal definition appears in Section \ref{sec:polymer} where the result will be proven.

\begin{lemma}
	\label{lem:iop_polymer} For $p \geq \f{2}{3}$,
\begin{align*}
    \fr{\Cnt - 2^{2^{d-1}} \Rnd{\sum_{S \in \Good^\Even} 2^{-\N(S)}
        + \sum_{S \in \Good^\Odd} 2^{-\N(S)}}}
        {2^{2^{d-1}} \sigma_p \exp(\mu_p)} \pto 0.
\end{align*}
\end{lemma}

\subsection{Reduction to exponentials via birthday problem estimates}

For this subsection, the joint behavior of $(\Phi_\Even, \Phi_\Odd)$ will not
be important, and so we will denote $\Phi \coloneqq \Phi_\Even$.
Everything stated here will also hold for $\Phi_\Odd$.
When expanding
\begin{align*}
    \Phi^m = \Rnd{\sum_{v \in \Even} \varphi_v}^m,
\end{align*}
each term corresponding to a set $S \in \Good^\Even$ with $|S| = m$ is counted
$m!$ times, but there are also additional terms which contain duplicate factors,
or factors corresponding to vertices which share a neighbor in $\Odd$.

Each term in the expansion can be interpreted as an (unnormalized)
probability of seeing a specific sample when sampling $m$ vertices from $\Even$
with replacement, according to a (random) probability distribution $\pi$, which
is given by $\pi(v) \propto \varphi_v$.
The sum of terms which don't correspond to any $S \in \Good^\Even$ is 
then the (unnormalized) probability that there is a collision in this sample,
meaning that two sampled vertices are equal or share a neighbor.
We denote this event by $\Col_m$.
Additionally, the normalization constant for $\pi$ is $\Phi$, and so
in symbols, we have
\begin{align*}
    m! \sum_{\substack{S \in \Good^\Even \\ |S| = m}} \prod_{v \in S} 2^{-\N(v)}
    &= \Phi^m \cdot \Rnd{1 - \pi(\Col_m)}.
\end{align*}
Relying on the strategy already outlined, 
we prove the following lemma, in Section \ref{sec:proof_of_iop_birthday}.
\begin{lemma}
    \label{lem:iop_birthday}
    Let $\zeta_p = 1$ for $p > \fr{2}{3}$ and $\zeta_{\f 23} = e^{-{\f 14}}$.
    We have
    \begin{align*}
        \fr{\sum_{S \in \Good^\Even} 2^{-\N(S)} - \zeta_p \exp(\Phi_\Even)}
        {\sigma_p \exp(\mu_p)} \pto 0,
    \end{align*}
    and the same holds if we replace $\Even$ by $\Odd$.
\end{lemma}

As promised, we now finish off the proof of Theorem 
\ref{thm:main}
\subsection{Proof of main theorem}
\label{sec:mainproof}

We will write $\fr{\Cnt}{2^{2^{d-1}} \sigma_p \exp(\mu_p)}$ as a small telescoping
sum, ending in a term whose distributional limit we understand, with all the
intermediate terms tending to zero in probability and consequently having no bearing on the distributional limit.
Specifically, we have

\begin{align*} 
    \fr{\Cnt}{2^{2^{d-1}} \sigma_p \exp(\mu_p)}
    &= \fr{\Cnt - 2^{2^{d-1}}
        \Rnd{\sum_{S \in \Good^\Even} 2^{-\N(S)}
        + \sum_{S \in \Good^\Odd} 2^{-\N(S)}}}
        {2^{2^{d-1}} \sigma_p \exp(\mu_p)}  \\
    \intertext{\hfill (converges to $0$ in probability by Lemma \ref{lem:iop_polymer})}
    &\qquad + \fr{\sum_{S \in \Good^\Even} 2^{-\N(S)} - \zeta_p \exp(\Phi_\Even)}
        {\sigma_p \exp(\mu_p)}
        + \fr{\sum_{S \in \Good^\Odd} 2^{-\N(S)} - \zeta_p \exp(\Phi_\Odd)}
        {\sigma_p \exp(\mu_p)} \\
    \intertext{\hfill (both terms converge to $0$ in probability by Lemma \ref{lem:iop_birthday})}
    &\qquad + \zeta_p \fr{\exp(\Phi_\Even) + \exp(\Phi_\Odd)}{\sigma_p \exp(\mu_p)},
\end{align*}
and the last expression has a distributional limit which is described in 
Lemma \ref{lem:iop_jointclt}.
For $p > \fr{2}{3}$, we $\zeta_p = 1$, and Lemma \ref{lem:iop_jointclt} tells us that
\begin{align*}
    \fr{\exp(\Phi_\Even) + \exp(\Phi_\Odd) - 2 \exp(\mu_p)}
    {\sqrt{2} \cdot \sigma_p \exp(\mu_p)}
    \weakto \Nor{0}{1}.
\end{align*}
Hence, for  $p>\fr{2}{3}$
\begin{align*}
    \fr{\Cnt - 2^{2^{d-1}} \cdot 2 \exp \mu_p}
    {2^{2^{d-1}} \cdot \sqrt{2} \cdot \sigma_p \exp(\mu_p)}
    \weakto \Nor{0}{1}.
\end{align*}
As for $p = \fr{2}{3}$, now we have $\zeta_p = e^{-{\f 14}}$ and
$\sigma_p = \fr{1}{\sqrt{2}}$.
Lemma \ref{lem:iop_jointclt} tells us that
\begin{align*}
    \fr{\exp(\Phi_\Even) + \exp(\Phi_\Odd)}{\exp(\mu_p)}
    \weakto \Exp{\fr{W_1}{\sqrt{2}}} + \Exp{\fr{W_2}{\sqrt{2}}},
\end{align*}
where $W_1, W_2 \overset{\text{i.i.d.}}{\sim} \Nor{0}{1}$, which finishes the proof.
\qed

\subsection{Organization of the article}\label{org34}
For the ease of navigation, we record here the general structure of the rest of the paper. In the upcoming Section \ref{sec:jointclt}, we prove the joint CLT result for $(\Phi_{\cE}, \Phi_{\cO}).$ In Section \ref{sec:polymer}, we include all the inputs from \cite{ks} relying on cluster expansion estimates and deliver the proof of the key approximation result, Lemma \ref{lem:iop_polymer}. In Section \ref{sec:birthday_full}, we address the alluded to birthday problem and prove Lemma \ref{lem:iop_birthday}. Section \ref{sec:geometry} is devoted to proving the structural result Theorem \ref{thm:geometry} and the validity of the sampling procedure outlined in Proposition \ref{lem:sampling1} and its counterpart for the case $p=\f23.$
Section \ref{hardcore3456} outlines how the arguments of the paper almost as is, with certain key estimates replaced by
their natural $\lambda-$counterpart, carries over to the general hard-core model. We end with the Appendix, where certain simple lemmas are proved for completeness.

\subsection{Acknowledgements} MBRC and SG were partially supported by NSF career grant 1945172.
VW was partially supported by the NSF Graduate Research Fellowship grant DGE 2146752.
SG learnt about this problem from Gal Kronenberg's talk on \cite{ks} while attending the workshop
titled `Bootstrap Percolation and its Applications' at Banff in April 2024. He thanks the speaker
as well as the organizers.

\section{A joint CLT result for singleton polymers via Stein's method}\label{sec:jointclt}
In this section we prove Lemma \ref{lem:iop_jointclt}.\\

Recall that the sums
\begin{align*}
    \Phi_\Even = \sum_{v \in \Even} \varphi_v
    \qquad \text{and} \qquad
    \Phi_\Odd = \sum_{v \in \Odd} \varphi_v
\end{align*}
 are not independent, because deleting one edge affects the neighborhood
of one odd vertex and one even vertex.
However, their dependence is sufficiently weak that we will be able to prove a joint
CLT for these two variables, which immediately implies Lemma \ref{lem:iop_jointclt}.

\begin{lemma}
\label{lem:jointclt}
For $\Phi_\Even$ and $\Phi_\Odd$ defined as above, $\mu_p = \fr{1}{2}(2-p)^d$, and
$\sigma_p^2 = \fr{1}{2}\Fr{4-3p}{2}^d$, we have
\begin{align*}
    \Rnd{
        \fr{\Phi_\Even - \mu_p}{\sigma_p}, \fr{\Phi_\Odd - \mu_p}{\sigma_p}
    } \weakto \Nor{0}{1} \otimes \Nor{0}{1}.
\end{align*}
\end{lemma}

While there might be multiple ways the above might be approached, it seems that Stein's method to prove CLT for random
variables with a sparse dependency graph (a graph with the random variables as the vertex set and an edge set such that
any vertex is independent of the set of all variables outside its neighborhood) is particularly well suited to our
situation. Towards this, we will invoke
{\cite{ross}*{Section 3.2} which was adapted from the main result of \cite{stein}.}

\begin{theorem}{\cite{ross}*{Section 3.2}}
\label{thm:dependentstein}
Let $Y_{n,1}, \dotsc, Y_{n,n}$ form a triangular array of random variables with $\E[{Y_{n,i}}^4] < \infty$,
$\E[Y_{n,i}] = 0$, $\sigma^2 = \Var(\sum_{i=1}^n Y_{n,i})$, and define $\widetilde{W} = \sum_{i=1}^n Y_{n,i} / \sigma$. 
Let the collection $(Y_{n,1}, \dotsc, Y_{n,n})$ have dependency neighborhoods $N_i$,
$i=1,\dotsc,n$, with $D \coloneqq \max_{1 \leq i \leq n} |N_i|$.
Then $\widetilde{W} \weakto W$ as $n\to \infty$ where {$W$ is a  standard normal random variable} if 
\begin{align}
    \label{eq:wassersteinbound}
     \fr{D^2}{\sigma^3} \sum_{i=1}^n \E|Y_{n,i}|^3
        + \fr{D^{3/2}}{\sigma^2} 
        \sqrt{\sum_{i=1}^n \E[Y_{n,i}^4]} \to 0.
\end{align}
\end{theorem}
In fact, as is often common with applications of Stein's method,  the result quoted from \cite{ross} is in terms of
Wasserstein distance of $\widetilde W$ from $W$. However for our purposes the above weak convergence statement will
suffice. 

Before we start the proof of Lemma \ref{lem:jointclt}, we state a few facts that
will be useful throughout the paper.

\begin{lemma}
\label{lem:xvmoments}
For $\varphi_v = 2^{-\N(v)} \sim 2^{- \Bin(d,p)}$ as above, we have
\begin{align*}
    \E[\varphi_v^k] = \Rnd{1 - \fr{2^k-1}{2^k} p}^d
\end{align*}
for any positive integer $k$.
\end{lemma}

\begin{proof}
Since $\Bin(d,p)$ is a sum of $d$ independent $\Ber(p)$ variables,
\begin{align*}
    \E[\varphi_v^k] &= \Rnd{\Ex{2^{-k \cdot \Ber(p)}}}^d \\
    &= \Rnd{(1-p) + \fr{1}{2^k} p}^d \\
    &= \Rnd{1 - \fr{2^k-1}{2^k}p}^d
\end{align*}
as required.
\end{proof}

\begin{lemma}
\label{lem:exvarcov}
For $\Phi_\Even$ and $\Phi_\Odd$ defined as above, we have
\begin{enumerate}
    \item $\E[\Phi_\Even] = \E[\Phi_\Odd] = \fr{1}{2}(2-p)^d$,
    \item $\Var(\Phi_\Even) = \Var(\Phi_\Odd) = \fr{1}{2}\Fr{4-3p}{2}^d(1-o(1))$,
    \item $\Cov(\Phi_\Even, \Phi_\Odd) \ll \Var(\Phi_\Even)$.
\end{enumerate}
Here and throughout the rest of the article we will use $a\ll b$ to denote that $\f ab$ exponentially decays in $d.$ 
\end{lemma}

\begin{proof}
First we calculate the expectations, which are equal by symmetry.
Since each $\varphi_v \sim 2^{- \Bin(d,p)}$, we have
\begin{align*}
    \E[\Phi_\Even] &= \sum_{v \in \Even} \E[\varphi_v] = 2^{d-1} \Rnd{1 - \fr{p}{2}}^d = \fr{1}{2} (2-p)^d.
\end{align*}
as required.

Next we calculate the variances, which are also equal by symmetry.
Since all the $\varphi_v$ for $v \in \Even$ are independent,
\begin{align*}
	\Var(\Phi_\Even) &= \sum_{v \in \Even} \Rnd{\E[\varphi_v^2] - \E[\varphi_v]^2} = 2^{d-1} \Rnd{\Rnd{1 - \fr{3p}{4}}^d -
	\Rnd{1 - \fr{p}{2}}^{2d}} = \fr{1}{2} \Rnd{\fr{4-3p}{2}}^d \Rnd{1 - o(1)},
\end{align*}
as $\Rnd{1 - \f p2}^2 < 1 - \f{3p}{4}$ for all $p \in (0, 1)$.\\

Finally, we turn to the covariance.
Since, for any  $v\in \cE, w\in \cO,$ $\varphi_v$ and $\varphi_w$ are independent when $v \not\sim w$, we have
\begin{align*}
    \Cov(\Phi_\Even, \Phi_\Odd)
    = \sum_{v \sim w} \Rnd{\E[\varphi_v \varphi_w] - \E[\varphi_v] \E[\varphi_w]}.
\end{align*}
Now, when $v \sim w$, there is exactly one edge connecting them, and there are
$2d-2$ edges which are incident to only one of the two vertices.
So $\varphi_v \varphi_w \sim 2^{-2 \cdot \Ber(p)} 2^{-\Bin(2d-2,p)}$, where the two random variables 
in the exponent are independent. 
So we have
\begin{align*}
    \Cov(\Phi_\Even, \Phi_\Odd)
    &= 2^{d-1} d \Rnd{
        \Rnd{1 - \fr{3p}{4}} \Rnd{1 - \fr{p}{2}}^{2d-2}
        - \Rnd{1 - \fr{p}{2}}^{2d} } \\
    &= 2^{d-1} d \Rnd{1 - \fr{p}{2}}^{2d-2} \Rnd{
        1 - \fr{3p}{4} - \Rnd{1 - \fr{p}{2}}^2 } \\
    &= d \Rnd{\fr{(2 - p)^2}{2}}^{d-1} \fr{p(1-p)}{4}.
\end{align*}
Again since $0 < p < 1$, we have $(2-p)^2 = 4 - 4p + p^2 < 4 - 4p + p = 4 - 3p$,
and so this is exponentially smaller than the variance.
\end{proof}

Given the above preparation, we are now in a position to finish the proof of Lemma \ref{lem:jointclt}.
\begin{proof}[Proof of Lemma \ref{lem:jointclt}]
By standard weak convergence theory for multivariate Gaussians, we will consider their arbitrary linear combinations to
reduce to a one-dimensional problem. Towards this, let $c_1, c_2$ be constants such that $c_1^2 + c_2^2 = 1$.
Define
\begin{align*}
    \tilde{\Phi} &\coloneqq c_1 (\Phi_\Even - \mu_p) + c_2 (\Phi_\Odd - \mu_p) \\
    &= \sum_{v \in \Even} c_1 (\varphi_v - \E[\varphi_v]) + \sum_{v \in \Odd} c_2 (\varphi_v - \E[\varphi_v]).
\end{align*}
We would like to apply Theorem \ref{thm:dependentstein} to show that
$\fr{\tilde{\Phi}}{\sigma_p} \weakto \Nor{0}{1}$.
We of course have
\begin{align*}
    \E[c_1 (\varphi_v - \E[\varphi_v])] = 0 \qquad \text{and} \qquad
    \E[c_2 (\varphi_v - \E[\varphi_v])] = 0,
\end{align*}
and also, invoking Lemma \ref{lem:exvarcov},
\begin{align*}
    \Var\Rnd{\tilde{\Phi}}
    &= c_1^2 \Var(\Phi_\Even) + c_2^2 \Var(\Phi_\Odd) 
        + 2 c_1 c_2 \Cov(\Phi_\Even, \Phi_\Odd) \\
    &= \fr{1}{2} \Fr{4-3p}{2}^d (1 + o(1)) \\
    &= \sigma_p^2 (1 + o(1)).
\end{align*}
So the conclusion of Theorem \ref{thm:dependentstein}, namely that
$\fr{\tilde{\Phi}}{\sqrt{\Var(\tilde{\Phi})}} \weakto \Nor{0}{1}$, implies
that $\fr{\tilde{\Phi}}{\sigma_p} \to \Nor{0}{1}$ as well.
It remains to check that \eqref{eq:wassersteinbound} holds.
The dependency neighborhood of any $c_i (\varphi_v - \E[\varphi_v])$ is exactly the neighborhood
in $\Q$ of $v$, so the size of all dependency neighborhoods is $d$.
Now let's calculate the moments of $\varphi_v - \E[\varphi_v]$.
Since $\varphi_v, \E[\varphi_v] \geq 0$, for any $k$ we have
\begin{align*}
    \Ex{\Abs{\varphi_v - \E[\varphi_v]}^k} &\leq \Ex{\Rnd{\varphi_v + \E[\varphi_v]}^k} \\
										   &= \sum_{j=0}^k \binom{k}{j} \Ex{\varphi_v^j} \Rnd{\E[\varphi_v]}^{k-j} \\
										   &\leq 2^k \max_{j=0, \ldots, k} \Ex{\varphi_v^j} \Rnd{\E[\varphi_v]}^{k-j} \\
										   &= 2^k \max_{j=0, \ldots, k} \Rnd{\Ex{\varphi_v^j}^{1/j}}^j
									   \Rnd{\Ex{\varphi_v}}^{k-j}.
\end{align*}
Now, by Jensens' inequality, $\E[\varphi_v]\le \Ex{\varphi_v^j}^{1/j} \le \Ex{\varphi_v^k}^{1/k}$ for any $j\le k$, implying

\begin{align*}
	\Ex{\Abs{\varphi_v - \E[\varphi_v]}^3} \lesssim \Ex{\varphi_v^3} = \Fr{8-7p}{8}^d ,
\end{align*}
and
\begin{align*}
	\Ex{\Rnd{\varphi_v - \E[\varphi_v]}^4} \lesssim \Ex{\varphi_v^4} = \Fr{16-15p}{16}^d,
\end{align*}
where the moments were computed in Lemma \ref{lem:xvmoments}.
So the upper bound in \eqref{eq:wassersteinbound} is, up to constants, at most
\begin{align*}
    \fr{d^2 2^d \Fr{8-7p}{8}^d}
        {\Fr{4-3p}{2}^{3d/2}} 
        + \fr{d^{3/2}\sqrt{2^d \Fr{16-15p}{16}^d}}
        {\Fr{4-3p}{2}^d}
    =
        d^2 \Fr{(8 - 7p)^2}{2(4-3p)^3}^{d/2}
        + d^{3/2} \Fr{16-15p}{2 (4-3p)^2}^{d/2}.
\end{align*}
Both terms above are exponentially small for $0 < p < 1$, as can be 
seen by a slightly tedious but routine calculation.
So $\fr{\tilde{\Phi}}{\sigma_p} \weakto \Nor{0}{1}$ for any $c_1, c_2$ with
$c_1^2 + c_2^2 = 1$.
This finishes the proof.
\end{proof}

\subsection{Proof of Lemma \ref{lem:iop_jointclt}}
\label{sec:proof_of_iop_jointclt}

As a direct consequence of the CLT for $(\Phi_\Even, \Phi_\Odd)$, we also obtain a
distributional limit for $\exp(\Phi_\Even) + \exp(\Phi_\Odd)$; for $p>\fr{2}{3}$
we get another normal random variable, and for $p=\fr{2}{3}$ we get a sum of 
two log-normal random variables.

In particular, we can prove Lemma \ref{lem:iop_jointclt}, whose statement has 
been reproduced below for the reader's convenience.

\begin{lemma}[Restatement of Lemma \ref{lem:iop_jointclt}]
Let $\mu_p = \fr{1}{2}(2-p)^d$ and $\sigma_p^2 = \fr{1}{2} \Fr{4-3p}{4}^d$.
If $p>\fr{2}{3}$, then 
\begin{align*}
    \fr{\exp(\Phi_\Even) + \exp(\Phi_\Odd) - 2 \exp(\mu_p)}
    {\sqrt{2} \cdot \sigma_p \exp(\mu_p)}
    \weakto \Nor{0}{1}.
\end{align*}
If $p=\fr{2}{3}$, then 
\begin{align*}
    \fr{\exp(\Phi_\Even) + \exp(\Phi_\Odd)}{\exp(\mu_p)} 
    \weakto \Exp{\fr{W_1}{\sqrt{2}}} + \Exp{\fr{W_2}{\sqrt{2}}},
\end{align*}
where $W_1$ and $W_2$ are i.i.d.\ $\Nor{0}{1}$.
\end{lemma}

\begin{proof}
Let $W_\Even = \fr{\Phi_\Even - \mu_p}{\sigma_p}$ and 
$W_\Odd = \fr{\Phi_\Even - \mu_p}{\sigma_p}$ so that by Lemma \ref{lem:jointclt},
\begin{align*}
    (W_\Even, W_\Odd) \weakto \Nor{0}{1} \otimes \Nor{0}{1}.
\end{align*}
When $p > \fr{2}{3}$, $\sigma_p \ll 1$, and so
\begin{align*}
    \exp(\Phi_\Even) + \exp(\Phi_\Odd)
    &= \exp(\mu_p + \sigma_p W_\Even) + \exp(\mu_p + \sigma_p W_\Odd) \\
    &= \exp(\mu_p) (1 + \sigma_p W_\Even') + \exp(\mu_p) (1 + \sigma_p W_\Odd') \\
    &= 2 \exp(\mu_p) + \sqrt{2} \cdot \sigma_p \exp(\mu_p)
        \fr{W_\Even' + W_\Odd'}{\sqrt{2}},
\end{align*}
where $W_\Even' = \sigma_p^{-1} (\exp(\sigma_p W_\Even) - 1)$ and
$W_\Odd' = \sigma_p^{-1}(\exp(\sigma_p W_\Odd) - 1)$.
If we can show that $W_\Even' - W_\Even \pto 0$ and $W_\Odd' - W_\Odd \pto 0$,
then $\fr{W_\Even' + W_\Odd'}{\sqrt{2}} \weakto \Nor{0}{1}$, which would yield the 
desired conclusion.

To do this, observe that Taylor's theorem tells us that for any $z \in \R$,
\begin{align*}
    \Abs{\fr{\exp(\sigma_p z) - 1 - \sigma_p z}{\sigma_p}} 
    \leq \fr{1}{\sigma_p} \fr{\sigma_p^2 z^2}{2}
        \sup_{|t| \in [0, \sigma_p |z|]} \exp(t)
    \leq \fr{\sigma_p z^2}{2} \exp(\sigma_p |z|).
\end{align*}
Now, for $z = W_\Even$ or $z = W_\Odd$ we have
$\sigma_p z^2 \pto 0$ and $\sigma_p |z| \pto 0$ since $z$ is tight.
Thus, by the continuous mapping theorem, 
$\fr{\sigma_p z^2}{2} \exp(\sigma_p |z|) \pto 0$ as well, finishing the proof
for $p > \fr{2}{3}$.

Now, when $p = \fr{2}{3}$, $\sigma_p^2 = \fr{1}{2}$, so
\begin{align*}
    \fr{\exp(\Phi_\Even) + \exp(\Phi_\Odd)}{\exp(\mu_p)}
    &= \fr{\Exp{\mu_p + \fr{W_\Even}{\sqrt{2}}} +
        \Exp{\mu_p + \fr{W_\Odd}{\sqrt{2}}}}{\exp(\mu_p)} \\
    &= \Exp{\fr{W_\Even}{\sqrt{2}}} + \Exp{\fr{W_\Odd}{\sqrt{2}}},
\end{align*}
which converges in distribution to the desired random variable by the continuous
mapping theorem.
\end{proof}
\section{Inputs from cluster expansion}

\label{sec:polymer}
In this section, we restate a precise version of as well as deliver the proof of Lemma \ref{lem:iop_polymer}. We first set up a convenient language and introduce the definitions of the central objects.
This section will primarily rely on estimates from \cite{ks}, or more streamlined versions thereof. Nonetheless, to
make the paper self-contained we will often include brief arguments to justify the estimates. The interested reader is
however encouraged to
review \cite{ks}*{Sections 3,4} for all the details.  We start with a few definitions.

\begin{definition}
    \label{def:omega}
    For any subset $S \sse \cE$ or $S\sse \cO$, and $B\sse N(S)$, define 
\begin{equation}\label{decpoly12}
    \om(S,B):= 2^{-N(S)} (1 - p)^{E(S, B)},
   \end{equation}
    where $E(S, B)$ is the number of edges between $S$ and $B$.
Thus, for $S\sse \Even,$ and similarly for $S\in \cO$, $\om(S,B)$ is the expected number of independent sets $I$ such
that $I\cap \cE=S$ and $I \cap \cO \cap N(S)=B$ (note that $N(S)\sse \cO$) normalized by dividing by $2^{|\cO|}$, the latter being the
total number of independent sets which are subsets of $\cO$, i.e., corresponding to $S=\varnothing.$
Further, define 
    \begin{equation}\label{weightmarginal} \om(S) = \sum_{B\sse N(S)}\om(S,B) . 
    \end{equation}
    
\end{definition}
Thus, $\omega(S)$ is the expected number of independent sets $I$ such that $I\cap \cE=S$, normalized as above. This also
leads to the following identity.   \begin{align}\label{weightpoly}
        \E 2^{-N_p(S)}  &= 2^{-N(S)} \cdot \Ex{2^{N(S) - N_p(S)}} 
                        = 2^{-N(S)} \cdot \Ex{\sum_{B \sse N(S)} \1{B \not\sse N_p(S)}} \\
                        \nonumber
                        &= 2^{-N(S)} \sum_{B \sse N(S)} \Ex{\1{B \not\sse N_p(S)}} 
                        = 2^{-N(S)} \sum_{B \sse N(S)} (1 - p)^{E(S, B)} \\
                        \nonumber
                        &= \om(S)
    \end{align}
Now, note that the number of independent sets $\Cnt$ is given exactly by
\begin{align}\label{count12}
    \Cnt = 2^{2^{d-1}} \sum_{S \sse \Even} 2^{-\N(S)},
\end{align}
since, if we restrict to considering independent sets which have $S$ as their 
intersection with the even side, then there are $2^{d-1} - \N(S)$ allowable 
vertices on the odd side, and we can take an arbitrary subset of them to obtain 
an independent set.

It will turn out that most independent sets will be such that their intersection with $\cE$ or $\cO$ will be quite
disparate in size. Consequently, as already alluded to, we will term the side with the smaller intersection as the
\emph{minority} side with
the  vertices on this side being considered as defects, and the other side being the \emph{majority} side. We will
denote the set of independent sets with the minority side being $\cE$ as $\EM$ and similarly for the odd side as $\OM.$

For brevity, we will mostly present our arguments for $\EM$ with symmetric arguments holding for $\OM.$
 We start with some central definitions; while we will primarily focus on the even side, analogous statements hold for the odd side as well.

Recall that two vertices $u, v \in \cE$ are said to be $2$-neighbors if they share a common
neighbor (necessarily in $\cO$). This naturally leads to the following notion of connectivity.

\begin{definition}[$2$-linked]
    \label{def:twolinked}
    A subset $S \sse \cE$ is said to be $2$-linked if it is connected under the edge-set
    determined by $2$-neighbors, i.e., any two vertices $u,v \in S$ admit a chain of vertices in $S$, starting and
    ending at $u$ and $v$ respectively, where each element is a $2$-neighbor of the preceding one.
\end{definition}

Next, we define closures of sets. 
\begin{definition}[Closure]
    \label{def:closure}
	For $S \subseteq \cE$, define its closure $\bar{S}$ to be the largest set $T \subseteq \cE$ such that $S$ and $T$
	have the same neighborhood in $\Q$, i.e., $\bar{S} = \{v \in \cE : N(v) \subseteq N(S)\}$. Note that trivially
	$|\bar{S}| \leq d^2 |S|$. However, owing to the iso-perimetric properties of $\Q$, a substantially stronger statement
	holds. Estimates recorded in \cite{ks}*{Lemma 1.8}, which turned out to be useful in many previous works as well, imply that as long as $|S|\le \frac{1}{2}2^{d-1},$ we have $|S|\le |N(S)|$, which implies that as long as $|N(S)|\le \frac{1}{2}2^{d-1},$ we have $|\overline S|\le |N(S)|.$
\end{definition}

\begin{definition}[Polymer]
    \label{def:polymer}
    A subset $\g \sse \cE$ is called a polymer if it is $2$-linked and has a closure of size at most $\f{3}{4} \cdot
    2^{d - 1}$.
\end{definition}

The number $\f34$ is somewhat of an arbitrary choice following \cite{ks} and as also remarked therein, the constant can be
chosen to be any number between $({\f 12},1).$ This constraint on the size of a polymer stems from the fact that sets which
occupy at least half of both $\cE$ and $\cO$ do not have very disparate minority and majority sides (one of
the undesirable properties listed in Section \ref{sec:idea}). Note however that they
span a lot of internal edges in the hypercube and hence the probabilistic cost of making them independent is quite high. We will record an estimate which will indicate that they will not play a prominent role and can be safely ignored.

\begin{definition}[Good sets]\label{goodset123}
A subset $S\sse \cE$ (similarly for $\cO$) will be called good if $|S|\le \frac{2^d}{d^2}$ and no two elements of $S$
are $2-$neighbors. The collection of such sets will be called $\Good^\Even$ (similarly $\Good^\Odd$).  The value
$\fr{2^d}{d^2}$ is not particularly special and falls out of certain estimates from \cite{ks} that we will serve as
inputs for our argument.  It is worth  remarking however that a polymer $S$ of size at most $\fr{2^d}{d^2}$
automatically has a closure of size at most $\frac{3}{4}2^{d-1}$ due to the bound stated above in Definition
\ref{def:closure}. 
\end{definition} 

We now restate Lemma \ref{lem:iop_polymer} before diving into its proof.

\begin{lemma}
\label{lem:iop_polymer1} For $p\ge \f23,$
\begin{align*}
    \fr{\Cnt - 2^{2^{d-1}} \Rnd{\sum_{S \in \Good^\Even} 2^{-\N(S)}
        + \sum_{S \in \Good^\Odd} 2^{-\N(S)}}}
        {2^{2^{d-1}} \sigma_p \exp(\mu_p)} \pto 0.
\end{align*}
\end{lemma}

The above will be a consequence of three statements. The first says that with high probability the majority of the
independent sets will
have at least one side whose closure satisfies the polymer bound of $\f{3}{4} \cdot
    2^{d - 1}$. 
This will imply that for such sets, at least one of the sides admits a decomposition into $2-$linked components each of
which is a polymer. The second statement will prove that among all such sets, most of them have clear majority and
minority sides, i.e. the side with the polymer bound will in fact be the minority side with size at most $\f{2^d}{d^2}$ while the majority side will be bigger than $\f{2^d}{d^2}$.  This is the formal definition of $\EM$ and $\OM$ which have already appeared in our discussion. The final statement will prove
that for most such independent sets, on the minority side,  the polymers are all \emph{singletons}. This final class of sets will be termed $\Nice$ and will be the only class that will feature throughout the rest of the paper.\\

Towards accomplishing the above, we first introduce the following definition.
\def\NPolyEnc{\mathsf{NonPolyEnc}}
\begin{align}\label{non-polymer}
\NPolyEnc \coloneq \Bra{I \sse \Q,\, \text{ such that } |\bar{I\cap \cO}|, |\bar{I\cap \cE}|\ge \frac{3}{4}2^{d-1}}.
\end{align}
The acronym denoting that no side is encodable as a union of polymers. Consequently, we will call the set of remaining  sets $\mathsf{PolyEnc}$.
\def\PolyEnc{\mathsf{PolyEnc}}
The reason behind constraining the size of  the closure and not just the set itself in the above definition is more of a
technical nature leading to desirable combinatorial estimates, and we will refrain from elaborating this further, instead
referring the interested reader to \cite{ks}*{Section 3}.

We then have the following bound.

\begin{lemma}\cite{ks}*{Lemma 4.16}\label{lem:nonpolysmall}
$\E |\NPolyEnc|\le \E [\Cnt] \cdot O(\exp(-2^d/d)).$
\end{lemma}
\emph{
Note that $\NPolyEnc$ is a deterministic class of sets. Whenever we use notation of the form $|\NPolyEnc|$ we mean the number of independent sets in $\Qp$ which belong to $\NPolyEnc$ (such a convention will also be adopted for other deterministic class of subsets of $\Q$).}\\

Now, for instance, by \eqref{expectation12} and Lemma \ref{lem:exvarcov},
\begin{align}\label{estimate56}
    \E[\Cnt]& = 2\cdot 2^{d-1}\cdot\exp\Rnd{\E[\Phi_{\Even}](1+o(1))}\\
    \nonumber
    &= 2 \cdot 2^{2^{d-1}} \cdot \Exp{\frac{1}{2}(2-p)^d (1+o(1))}\text{ and,}\\
     \sqrt{\Var(\Phi_{\Even})} &= \frac{1 - o(1)}{\sqrt{2}} (2 - 3p/2)^{d/2}.
\end{align}
Thus, on an application of Markov's inequality, with probability at least $1- \Exp{-\frac{2^{d}}{5d}},$
\begin{equation}\label{nonpolysmall}
\fr{\Abs{\NPolyEnc}}{2 \cdot 2^{2^{d-1}} \Exp{\E\Phi_{\Even}} \sqrt{\Var(\Phi_{\Even})}} \le  \Exp{-\frac{2^{d}}{5d}}.
\end{equation}

Now, by definition, every independent set  in $\PolyEnc$ admits a decomposition of its intersection either with
$\cE$ and $\cO$ into polymers.  The remainder of the section delivers the two remaining statements promised.  The first
shows that with high probability, for the majority of the independent sets in $\PolyEnc$, there is an unambiguous
minority or defect side and a majority side where the size of the minority side is less than $\frac{2^d}{d^2}$ while for
the majority side it is larger than that.

\renewcommand\OSmall{\mathsf{Small}^{\cO}}
\renewcommand\ESmall{\mathsf{Small}^{\cE}}

\subsection{Nearly every independent set has a minority side and a majority side}

At this point recalling $\EM$ and $\OM,$ we state the following lemma analogous to \eqref{nonpolysmall}.

\begin{lemma} With probability at least $1- \Exp{-\fr{2^d}{5d^4}}$
\label{lem:bigsmallprim}
\begin{align*}
{   \fr{|\PolyEnc|-\Abs{\EM}-\Abs{\OM}}{2 \cdot 2^{2^{d-1}} \Exp{\E\Phi} \sqrt{\Var(\Phi)}}} \le \exp{\left(-{\fr{2^d}{5d^4}}\right)}
\end{align*}
where $\Phi$ is either $\Phi_{\cE}$ or $\Phi_{\cO}.$
In particular, the above quantities converge to zero in probability.
\end{lemma}

Note that by definition, $\EM$ and $\OM$ are disjoint subsets of $\PolyEnc$, so the LHS is non-negative. 

\begin{proof} Note that that $\Cnt=|\PolyEnc|+|\NPolyEnc|.$
Thus using Lemma \ref{lem:nonpolysmall}, \eqref{estimate56} and Markov's inequality, it will suffice to show that
\begin{align}\label{probbound123}
	\fr{\E|\PolyEnc|-\Rnd{\E\Abs{\EM}+\E\Abs{\OM}}}{\E\abs{\PolyEnc}} \le \exp{\left(-{\fr{2^d}{d^4}}\right)}.
\end{align}

The above bounds now  follow from the estimates in \cite{ks}*{Section 3} which we review briefly next.  
As already referred to in the introduction, the main approach in \cite{ks} is to bound the expectation of the number of independent sets in $\PolyEnc$ which are neither in $\EM$ or in $\OM$, using cluster expansion.

The argument proceeds by analyzing a probability measure $\hat \mu$ defined on sets where the probability of a set is proportional to the probability that  the set is independent in $\Qp.$ (The actual measure used in \cite{ks} is slightly more complicated but we present a simplified version which will suffice to convey the main ideas.)
For any set of the form $I=A\cup B_1\cup B_2$ denoted by $(A,B_1,B_2)$, where $A= I \cap \cE$ and whose closure is at most $\f34 2^{d-1}$, $B_1\sse N(A),$ and
$B_2=I\cap\{\cO\setminus N(A)\}$ 
define
\begin{align}\label{measuredef}
\hat\mu_{\cE} (A,B_1,B_2)=
\frac{(1-p)^{E(A,B_1)}}{\hat Z_{\cE}},
\end{align}
($\hat\mu_{\cO}$ and $\hat Z_{\cO}$ are similarly defined). Thus $\hat Z_{\cE}$ is simply the expected number of independent sets whose even side has closure at most $\f342^{d-1}$ and hence is polymer encodable, which by symmetry is equal to $\hat Z_{\cO}.$
 Note that the marginal on the first two coordinates is proportional to the expected number of independent sets whose
even (or odd) side is $A$, and the odd (or even) side restricted to $N(A)$ is $B_1$ and hence $$\hat\mu_{\cE} (A,B_1)
= \frac{2^{2^{d-1}}}{\hat Z_{\cE}}\omega(A,B_1),$$ where the latter was defined in \eqref{decpoly12}. 
Now note that $\EM$ and $\OM$ are disjoint classes of sets in $\PolyEnc$. Further for any $I\in \PolyEnc$, $$\P(I \text{ is independent})\le \hat Z_{\cE}\hat\mu_{\cE}(I)+\hat Z_{\cO}\hat\mu_{\cO}(I).$$ 
Thus,
\begin{align}\label{sandwiching}
\hat Z_{\cE} \hat\mu_{\cE}(\EM)&+\hat Z_{\cO}\hat\mu_{\cO}(\OM)\le \E|\PolyEnc|\le \hat Z_{\cE}+\hat Z_{\cO}.
\end{align}
Now letting $\UD=\PolyEnc\setminus \Box{\EM\cup \OM},$ by 
\cite{ks}*{Lemma 3.4}, 
\begin{align}\label{probbound}
\hat \mu_{\cE}(\UD)=\hat \mu_{\cO}(\UD) \leq \Exp{- \fr{2^d}{d^4}}.
\end{align}
Note that the above follows if 
\begin{equation}\label{key765}
\hat \mu_{\cE}(\PolyEnc\setminus \EM)\le  \Exp{- \fr{2^d}{d^4}}
\end{equation}
 which is what \cite{ks } shows. The first step is to use combinatorial estimates of $\omega(A,B_1)$ defined in \eqref{decpoly12} to show that in \eqref{measuredef},  it is likely that $|A|\le \frac{2^d}{d^2}$ (a particular instance of such an estimate which will feature again
in our arguments is recorded shortly in Lemma \ref{weightestimate43}). 
Given this, since
$B_2$ is a uniformly chosen subset of  $\cO \setminus N(A)$ (under $\hat \mu_{\cE}$) and $|N(A)|\le \frac{2^d}{d},$ with probability at least
$1-\exp(-c 2^d),$ we have $|B_2|\ge \frac{2^{d}}{3}$. This finishes the proof of \eqref{probbound}.

\eqref{probbound123} is now an immediate consequence of \eqref{sandwiching} and \eqref{key765}.
\end{proof}

\def\SP{\mathsf{Ideal}}

\subsection{Every polymer is a singleton}
Our final statement of this section shows that all but a negligible fraction of  $\EM$ (and similarly $\OM$) are sets
whose defect side admits a polymer decomposition solely consisting of singletons, a class we have termed as $\SP.$

\begin{lemma}\label{notwopoly}
\begin{align}\label{notwopoly1}
\frac{\E\Abs{(\EM\cup \OM)\setminus \SP}}{\exp(\E\Phi_{\cE})\sqrt{
\Var(\Phi)}} \le \exp (-cd),
\end{align}
for some constant $c>0.$
\end{lemma}
\begin{proof}
For $I \in \EM$ (and similarly for $\OM$), let  $S=I\cap \Even$ denote the intersection of $I$ with the minority
side. As already defined, let $\Good^\cE$ be the set of such subsets which do not admit any polymer of size at least $2$. Let $\Bad^{\cE}$ be the
remaining ones. Note that the expected number of independent sets with minority side as $S$ is at most $2^{2^{d-1}}
\Ex{2^{-N_p(S)}}.$ Thus, all that remains to be done is to control $\Ex{\sum_{S \in \Bad^{\cE}} 2^{-N_p(S)}}$ using
\cite{ks}*{Theorem 4.1}.

To that end observe that for every set $S$ with $|S| \leq \frac{2^d}{d^2}$,
\begin{align*}
	\Ex{2^{-N_p(S)}} = \Ex{2^{-\sum_{i = 1}^r N_p(\g_i)}} = \prod_{i = 1}^r \Ex{2^{-N_p(\g_i)}} = \prod_{i = 1}^r \om(\g_i)
\end{align*}
where $\{\g_1, \ldots, \g_r\}$  form the polymers in $S$, since for a polymer $\g$, by \eqref{weightpoly}, $\om(\g) =
\Ex{2^{-N_p(\g)}}$.

Therefore,
\begin{align*}
	\sum_{S\sse \cE, |S|\leq 2^d/d^2} \Ex{2^{-N_p(S)}} = \sum_{S\sse \cE, |S|\leq 2^d/d^2} \prod_{i=1}^{r}\om(\g_i) \leq \Exp{\sum_{\g} \om(\g)}
\end{align*}
where the sum is over all polymers $\g$ which are subsets of $\cE$ (we will not specify this further to maintain brevity).
To see why this inequality holds, observe that for a fixed set
$S$, if the number of polymers it admits is $r$, then the $r^{th}$ term in the series expansion of $\Exp{\sum_\g \om(\g)}$ is
$\frac{1}{r!}\Rnd{\sum \om(\g)}^r$ contains a term like $ \prod_{i=1}^{r}\om(\g_i)$. Further, this term also
has a prefactor of $1$, since each of the $r!$ orderings of the polymer set appear once in the expansion. Finally, all
the terms in the expansion are non-negative and thus they do not cause any undesirable cancelation.

We will now use this connection to eliminate $\Good^{\cE}$ subsets so that we can bound the sum on $\Bad^{\cE}$
subsets. Observe that  similarly to the above,
\begin{align*}
    \sum_{S\sse \Bad^{\cE}} \Ex{2^{-N_p(S)}} &= \sum_{S \sse \Bad^{\cE}} \prod_{i=1}^r \om(\g_i) \\
                                    &\leq \Exp{\sum_\g \om(\g)} -
    \Exp{\sum_{|\g| = 1} \om(\g)}
\end{align*}
where $|\g| = 1$ indicates that $\g$ is a {\it singleton} polymer. To argue this, note that each bad set $S$ occurs in the expansion
of $\Exp{\sum_\g \om(\g)}$, but does not occur in the expansion of $\Exp{\sum_{|\g| = 1} \om(\g)}$ since at least one
polymer in its decomposition is of size at least 2 by definition. Finally, every term appearing in  $\Exp{\sum_{|\g| =
1}\om(\g)}$ also appears with the exact same coefficient in  $\Exp{\sum_\g \om(\g)}.$

Now
\begin{align*}
    \Exp{\sum_\g \om(\g)} - \Exp{\sum_{|\g| = 1} \om(\g)} &= \Exp{\sum_{|\g| = 1} \om(\g)} \Box{\Exp{\sum_{|\g| \geq 2} \om(\g)} - 1}
\end{align*}
and noticing that 
\begin{align*}
 {   \sum_{\g \sse \cE, |\g| = 1} \om(\g) = \Ex{\sum_{v\in \cE} 2^{-N_p(v)}} = \E \Phi_{\cE}},
\end{align*}
we have
\begin{align*}
	\Ex{\sum_{S \in \Bad^{\cE}} 2^{-N_p(S)}} \leq \Exp{\E \Phi_{\cE}} \Box{\Exp{\sum_{|\g| \geq 2} \om(\g)} - 1}.
\end{align*}

We now quote from \cite{ks} estimates of $\omega(\g)$ which also form the backbone of the bound
in \eqref{probbound}, simplified to suit our application. 
\begin{lemma}\cite{ks}*{Theorem 4.1}\label{weightestimate43}
\begin{align*}
    \sum_{\gamma: |\gamma| \geq 2} \om(\gamma) \Exp{|\gamma| d^{-3/2}} \leq d^{-3/2} 2^d \Exp{-g_2}, \text{ where }\\
    g_2 = (2d - 12)\log(1/(1 - p/2)) - 14\log d.
\end{align*}
 \end{lemma}

Plugging in, we get, 
\begin{align}
\label{estimate15}
    \sum_{|\g| \geq 2} \om(\g) \leq 2^d (1 - p/2)^{2d} \cdot \Bra{(1 - p/2)^{-12} d^{14}} 
\end{align}
Recalling the expression for $\Var(\Phi_{\cE}),$ from Lemma \ref{lem:exvarcov},
note that $2^d (1 - p/2)^{2d} \le e^{-cd} (2 -
    3p/2)^{d/2}$
as long as $p > 0.51$
since
\begin{align*}
    4(1 - p/2)^4 < 2 - 3p/2, \quad \forall p > 0.51.
\end{align*}
Further, since for all $p>2-\sqrt {2}$, 
\begin{align}
\sum_{\g: |\g|\ge 2}\om(\g)&=o(1), \text{it follows that}\\
\Box{\Exp{\sum_{|\g| \geq 2} \om(\g)} - 1}&= O\Rnd{\sum_{|\g| \geq 2} \om(\g)}.
\end{align}
Thus,  we get, 
\begin{align}
    \fr{\E\left[{\sum_{S \in \Bad} 2^{-N_p(S)}}\right]}{\Exp{\E \Phi_{\Even}}\sqrt{\Var(\Phi)}} =O(\exp(-cd)), \label{eq:badnegligible}
\end{align}
for some $c>0$ and hence this finishes the proof.
\end{proof}
We now finish the proof of Lemma \ref{lem:iop_polymer1}.

\begin{proof}
Since $\Cnt=|\PolyEnc|+|\NPolyEnc|$, by \eqref{nonpolysmall}, it suffices to prove
\begin{align*}
    \fr{|\PolyEnc| - 2^{2^{d-1}} \Rnd{\sum_{S \in \Good^\Even} 2^{-\N(S)}
        + \sum_{S \in \Good^\Odd} 2^{-\N(S)}}}
        {2^{2^{d-1}} \sigma_p \exp(\mu_p)} \pto 0.
\end{align*}
Now note that $|\EM|\le 2^{2^{d-1}} \sum_{S \in \Good^\Even} 2^{-\N(S)} +|\EM\setminus \SP|$ since for any independent set in $\EM\cap \SP$, the intersection with $\cE$ is necessarily in $\Good^\cE,$ with a similar statement holding for $\OM.$ 
Further, any independent set appearing in both sums, must necessarily be in $\UD$ since its intersections with both $\cE$ and $\cO$ must have sizes smaller than $\fr{2^d}{d^2}.$ Thus 
\begin{align}\label{imp546}
|\EM|+|\OM| &\le 2^{2^{d-1}}\left[ \sum_{S \in \Good^\Even} 2^{-\N(S)}+ \sum_{S \in \Good^\Odd} 2^{-\N(S)}\right]\\
&\,\,\,\,\,\,\,\,\,\,\,\,\,\,\,\,\,\,\,\,+ |{(\EM\cup \OM)\setminus \SP}|\\
\nonumber
&\le |\EM|+|\OM|+2|\UD|.
\end{align}
which is the same as saying
\begin{align}
|\PolyEnc|-|\UD| &\le 2^{2^{d-1}} \left[\sum_{S \in \Good^\Even} 2^{-\N(S)}+ \sum_{S \in \Good^\Odd} 2^{-\N(S)}\right]\\
&\,\,\,\,\,\,\,\,\,\,+ |{(\EM\cup \OM)\setminus \SP}|\\
\nonumber
&\le |\PolyEnc|+|\UD|.
\end{align}
This along with Lemmas \ref{lem:bigsmallprim} and \ref{notwopoly} finishes the proof.
\end{proof}

\begin{remark}\label{dettorandom}Note that bounds in this section have been proved for sizes of certain sets compared to
	$\exp(\E\Phi_{\cE})\sqrt{\Var(\Phi)}$, as this is what will be needed to prove Theorem \ref{thm:main}. The latter then implies
	that $\Cnt$ and $\exp(\E\Phi_{\cE})$ are within $1+o(1)$ of each other for $p>\f23$ while they are comparable for
	$p=\f23.$ Thus, in both cases all the results in this section continue to hold when the sizes are compared to $\Cnt$
	instead. We particularly record the version of Lemma \ref{notwopoly1} which will be useful in the analysis of the
	structure of a typical independent set in Section \ref{sec:geometry}.
\begin{align}\label{structureinput2}
    1-\fr{\abs{\SP}}
        {\Cnt} \pto 0.
\end{align}
\end{remark}

\section{Analysis of a random non-uniform birthday problem}
\label{sec:birthday_full}
This section is devoted to establishing Lemma \ref{lem:iop_birthday} which we recall first. Throughout this section, the
joint behavior of $(\Phi_\Even, \Phi_\Odd)$ is
unimportant.
So in this section we abbreviate $\Phi \coloneqq \Phi_\Even$, and note
that every result stated also holds for $\Phi_\Odd$.
\begin{lemma}
    \label{lem:iop_birthday234}
    Let $\zeta_p = 1$ for $p > \fr{2}{3}$ and $\zeta_{\f 23} = e^{-{\f 14}}$.
    We have
    \begin{align*}
        \fr{\sum_{S \in \Good^\Even} 2^{-\N(S)} - \zeta_p \exp(\Phi_\Even)}
        {\sigma_p \exp(\mu_p)} \pto 0,
    \end{align*}
    and the same holds if we replace $\Even$ by $\Odd$.
    As a consequence, we also have the following estimate:
    \begin{equation}\label{ratio987}
    \fr{\sum_{S \in \Good^\Even} 2^{-\N(S)}}{\zeta_p \exp(\Phi_\Even)} \pto 1, 
    \end{equation}
    and similarly for the odd side. 
\end{lemma}

As outlined in Section \ref{sec:idea}, this can be recast as a birthday problem but we first prove that it suffices to
restrict to sets $S$ of size within a certain window which will be followed by showing that the non-collision
probability is essentially $\zeta_p$ as long as the number of samples drawn is within that window.

\subsection{Relevant values of $m$}
\label{sec:relevant}
We define $\Good^\Even_m$ to be the collection of sets $S \in \Good^\Even$ with $|S| = m$,
 and $\Win = \Bra{m : |m - \mu_p| \leq \mu_p^{0.6}}$.
\begin{lemma}
\label{lem:relevantalpha}
Suppose $p \geq \fr{2}{3}$.
Then
\begin{align}
    \label{eq:alphasforpsi}
    \fr{\sum_{m \notin \Win}\fr{\Phi^m}{m!}}
    {\sigma_p \exp(\mu_p)} \pto 0.
\end{align}
Additionally, since
$\fr{\Phi^m}{m!} \geq \sum_{S \in \Good^\Even_m} 2^{-\N(S)}$
for every $m$, we also have
\begin{align}
    \label{eq:alphasforgood}
    \fr{\sum_{m \notin \Win} \sum_{S \in \Good^\Even_m} 2^{-\N(S)}}
    {\sigma_p \exp(\mu_p)} \pto 0.
\end{align}
\end{lemma}

Thus, the above lemma shows that we only need to consider values of $m$ within
a slightly larger than $\sqrt{\mu_p}$-sized window around $\mu_p$. This stems from the fact that a Poisson random
variable with mean $\theta$ (which in our case will be $\Phi$) is concentrated around $\theta$ at scale $\sqrt{\theta}.$
We record a concentration result below (the proof is provided in the Appendix; see Lemma \ref{lem:poissontails}) that
will imply the above statement. There exists $c>0$ such that, if $M \sim \Pois(\theta)$ where $\theta\ge 1, $  and 
    $t \leq c \sqrt{\theta}$, then    \begin{align}\label{poissontail456}
		\Px{|M - \lambda| > t \sqrt{\lambda}} \lesssim \exp\Rnd{- \fr{t^2}{3}}.
    \end{align}

\begin{proof}[Proof of Lemma \ref{eq:alphasforpsi}]
Choose $\eps=\frac{1}{{\mu_p}^{.01}}.$
By Chebyshev's inequality, with probability at least $1-\eps$, we have 
$|\Phi - \mu_p| \leq \fr{\sigma_p}{\sqrt{\eps}}$.
On this event, $|m - \mu_p| > \mu_p^{0.6}$ implies that
$|m - \Phi| > \fr{1}{4} \mu_p^{0.1} \sqrt{\Phi}$ for $d$ large enough,
since $\mu_p \gg 1 \geq \sigma_p$ for $p \geq \fr{2}{3}$.
Applying \eqref{poissontail456} we obtain,
\begin{align*}
    \sum_{m \notin \Win} \fr{\Phi^m}{m!}
    &\leq \sum_{m : |m - \Phi| > \fr{1}{4} \mu_p^{0.1} \sqrt{\Phi}}
        \fr{\Phi^m}{m!} \\
    &= \exp(\Phi) \cdot \Px{\Abs{\Pois(\Phi) - \Phi} 
        > \fr{1}{4} \mu_p^{0.1}\sqrt{\Phi}} \\
    &\lesssim \exp(\Phi) \cdot 2 \exp\Rnd{- \fr{\Rnd{\fr{1}{4} \mu_p^{0.1}}^2}{3}} \\
    &\lesssim \Exp{\mu_p + \fr{\sigma_p}{\sqrt{\eps}} - \fr{\mu_p^{0.2}}{48}}\\
    &\lesssim \Exp{\mu_p - \fr{\mu_p^{0.2}}{50}}
\end{align*}
for all large $d,$ with probability at least $1 - \eps$, where in the last inequality we used that
$\fr{\sigma_p}{\sqrt{\eps}}\ll \mu_p^{0.2}$ since $\mu_p = \exp(c_p d)$ and $1 \geq \sigma_p \geq \exp(-C_p d)$
for some constants $c_p, C_p > 0.$ This implies 
\begin{align*}
    \fr{\sum_{m \notin \Win} \fr{\Phi^m}{m!}}{\sigma_p \exp(\mu_p)}
    &\lesssim \Exp{-\fr{\mu_p^{0.2}}{100}}
\end{align*}
for large enough $d$, with probability at least $1-\eps$.
\end{proof}

\subsection{Poisson approximation in the birthday problem}
\label{sec:poisson}

The primary result of this section is Lemma \ref{lem:poisson}, which 
investigates the collision probabilities in the birthday problem, as indicated in Section \ref{sec:idea}. 
Let $\pi\coloneq \pi_{\varphi}$ be the random measure on $\Even$ satisfying $\pi(v)\propto \varphi_v$ and let
$\xi_1,\xi_2,\ldots, \xi_m$ be $m$ i.i.d. samples drawn from $\pi.$ Let the events $\Col_n$ and $\Rep_n$ be defined as
    \begin{align*}
        \Col_n &= \{ \exists 1 \leq i < j \leq n : \xi_i = \xi_j, \text{ or } \xi_i \sim \xi_j \},\\
        \Rep_n &= \{ \exists 1 \leq i < j \leq n : \xi_i = \xi_j \},
    \end{align*}
where $u \sim_2 v$ if $u$ and $v$ are 2-neighbors in $Q_d$ (as a convention, $u \not\sim_2 u$). We shall also be
interested in the counts:
\begin{align}\label{impquant23}
    N_\Col = \sum_{1 \leq i < j \leq n} \1{\xi_i = \xi_j, \text{ or } \xi_i \sim_2 \xi_j}, 
    N_\Rep = \sum_{1 \leq i < j \leq n} \1{\xi_i = \xi_j}, \text{ and }
    N_\Nbr= \sum_{1 \leq i < j \leq n} \1{\xi_i \sim_2 \xi_j}.
\end{align}
Note that these are random variables whose distributions are measurable with respect to the random variables
$\{\varphi_v\}$.  The most important of these, and the subject of the following lemma, is the law of $N_\Col$ which we
denote by $\cL_{n, d}$. The lemma states that with probability close to $1$, the variables $\{\varphi_v\}$ are such that,
$\cL_{n, d}$ is exponentially close, in total variation distance, to a Poisson random variable of parameter $1$ or
$\f14$ depending on whether $p>\f23$ or $p=\f23$ respectively for a certain class of values of $n.$

\begin{lemma}
    \label{lem:poisson}There exists $c=c_p>0,$ such that, with probability at least $1-\exp{\left(-cd\right)},$ the following events hold.\\

\noindent
 $\bullet$  For $p = \f 23$, 
    \[
        \TV{\cL_{n, d} - \Pois\Rnd{\f 14}} \le \exp{\left(-cd\right)}
    \]
    and for $p > \f 23$,
\[
        \TV{\cL_{n, d} - \del_0} \le \exp{\left(-cd\right)},
\]
as $d \to \infty$ as long as $n \in \Win$. \\

As a consequence, for $n$ as above, with probability at least $1-\exp{\left(-cd\right)},$\\
    
\noindent
$\bullet$   $|(1-\pi(\Col_n)) - e^{-{\f 14}}|$ for $p=\f23$ and $|\pi(\Col_n) |$ when $p > \f23$, are bounded
by $\exp{(-cd)}$ respectively. In the latter case, an explicit bound is also given by $\frac{d^2n^2}{\Phi^2}\sum_{v\in
\cE}\varphi_v^2.$ 
\end{lemma}
By symmetry, the natural counterpart of the above statement holds for the odd side.

Again, as in the proof of the CLT result in Lemma \ref{lem:jointclt}, in the presence of weak dependence, Stein's
method, now for \emph{Poisson convergence} seems particularly suitable. More precisely, we will invoke the following
quantitative total-variation bound, derived via the Stein-Chen technique. For convenience, we reproduce the required
statement from \cite{ross}. A more comprehensive treatment of the result may be found there, or in the survey article
\cite{chatterjee}. 
\begin{lemma}[\cite{ross}*{Theorem 4.7}]
    \label{lem:stein}
	Let $\chi_1, \ldots, \chi_n$ be indicator variables with $\P(\chi_i) = 1 = p_i, \widetilde{M} = \sum_i \chi_i$, and
	$\t = \E \widetilde{M} = \sum_i
    p_i$. For each $i$, let $\cN(i) \sse \{1, \ldots, n\}$ be the {\it dependency neighborhood}, i.e., $\chi_i$ is
    independent of $\{\chi_j : j \notin \cN(i)\}$. If $p_{ij} \df \E[\chi_i \chi_j]$ and $M \sim \Pois(\t)$, then 
    \[
		\TV{\widetilde{M} - M} \leq \min(1, \t^{-1}) \Rnd{\sum_{i = 1}^n\sum_{j \in
        \cN(i) \setminus \{i\}} p_{ij}+ \sum_{i = 1}^n \sum_{j \in \cN(i)} p_i p_j}.
    \]
\end{lemma}

\begin{proof}[Proof of Lemma \ref{lem:poisson}]
    \def\Cond{\mathsf{Condition}}
    \def\Joints{\mathsf{Joints}}
    \def\Products{\mathsf{Products}}
     Starting with the observation that $N_\Col = N_\Rep + N_\Nbr$, our first goal is to assert that the claimed
	asymptotic behavior of $N_\Col$ coincides with that of $N_\Rep$, by showing that $N_\Nbr$ is much smaller simply by
	an expectation bound. 

    To see this, observe that 
    \begin{align}
		\E_{\pi} N_\Nbr &\lesssim n^2 \cdot \pi(\xi_1 \sim_2 \xi_2) \lesssim n^2 \cdot \fr{\sum_{u \sim_2 v} \varphi_u
		\varphi_v}{\Rnd{\sum \varphi_v}^2} \eqcolon Y^{(1)}_d.
    \end{align}
    Then, if $Y^{(1)}_d \to 0$,
    \[
        \TV{N_\Nbr - \del_0} = \pi(N_\Nbr > 0) \leq \E_{\pi} N_\Nbr \lesssim Y^{(1)}_d \pto 0,
    \]
	 Recall that $n$ is assumed to be in $\Win$ and hence by concentration of $\Phi=\sum_{v\in \cE} \varphi_v$, the term
	 $\frac{n^{2}}{\Phi^2}$ will converge in probability to one and thus this is equivalent to showing that  $\sum_{u
	 \sim_2 v} \varphi_u \varphi_v$ converges to zero in probability.  This will follow again by computing its expectation,
	 the rationale being that it is quite unlikely for  $\varphi_u$ and $\varphi_v$ to be large when $u\sim v.$ The
	 computations are presented at the end of the proof.
 
    Turning to $N_\Rep$, let us first consider the case of $p = \f 23$. 
    We will apply Lemma \ref{lem:stein} to the random variables $\chi_{ij} \df \1{\xi_i = \xi_j}$,
    so that $N_\Rep = \sum_{i < j} \chi_{ij}$. The dependency structure is as follows: for fixed $i < j$, $\chi_{ij}$ is
    independent of $\{\chi_{k\ell}\}_{\{k, \ell\}}$, where the index is over all $\{k, \ell\}$ which do not overlap with
    $\{i, j\}$, i.e., $\{k, \ell\} \cap \{i, j\} = \varnothing$. The parameter $\t$ for our
    approximating $\Pois$, at $d$ (and $n$,
    which is determined by $d$)
    is then given by
    \begin{align}
        \label{eq:thetadef}
        \t = \t_d = \sum_{i < j} \pi(\chi_{ij} = 1) = \sum_{i < j} \pi(\xi_i = \xi_j) = \binom{n}{2} \fr{\sum \varphi_v^2}{(\sum \varphi_v)^2}.
    \end{align}
    
    Lemma \ref{lem:stein} describes the total-variation error in this approximation as 
    \[
        \TV{\cL_{n, d} - \Pois(\t)} \leq \min(1, \t^{-1}) \Rnd{{\mathrm{I}} + {\mathrm{II}}}
    \]
    where ${\mathrm{I}}$ and ${\mathrm{II}}$ are the two terms in the error bound, respectively. We compute them now:
    \begin{align*}
        {\mathrm{I}} = \sum_{\substack{i < j, k < \ell \\ (k, \ell) \in \cN((i, j)) \setminus \{(i, j)\}}} \pi(\chi_{ij} = 1,
        \chi_{k\ell} = 1)
                \lesssim n^3 \cdot \pi(\xi_1 = \xi_2 = \xi_3) 
                \lesssim n^3 \fr{\sum \varphi_v^3}{(\sum \varphi_v)^3} \df Y^{(2)}_d,
    \end{align*}
    and 
    \begin{align*}
        {\mathrm{II}} = \sum_{\substack{i < j, k < \ell \\ (k, \ell) \in \cN((i, j)) }} \pi(\chi_{ij} = 1)\pi(\chi_{k\ell}
        = 1) 
                  \lesssim n^3 \cdot \pi(\xi_1 = \xi_2) \pi(\xi_2 = \xi_3) 
                  \lesssim n^3 \Fr{\sum \varphi_v^2}{(\sum \varphi_v)^2}^2 \df Y^{(3)}_d.
    \end{align*}

    In summary, 
    \begin{itemize}
        \item if $Y^{(1)}_d \pto 0$, $\TV{N_\Nbr - \del_0} \pto 0$, 
        \item if $Y^{(2)}_d, Y^{(3)}_d \pto 0$, $\TV{N_\Rep - \Pois(\t)} \pto 0$,
		\item if $\t \pto \t_0$ for some $\t_0 > 0$, then $\TV{\Pois(\t_d) - \Pois(\t_0)} \pto 0$ (This is standard. A
			bound for concreteness is recorded in the appendix, as Lemma
            \ref{lem:poissontv}).
    \end{itemize}

    Thus, if all four convergences hold, triangle inequality for total-variation implies 
    \[
        \TV{N_\Col - \Pois(\t_0)} \pto 0,
    \]
    finishing the proof, at least for $p = \f 23$. We now proceed to implement this. 
    First observe that by Chebyshev's inequality, $\Phi = \sum_v \varphi_v$ satisfies that with probability at least $1-\exp{(-cd)},$
\begin{align}\label{compar2344}
       \left|\fr{\Phi}{\Ep}- 1\right|\le \exp{(-cd)}, 
    \end{align}
    since $\sqrt{\Var(\Phi)}$ is exponentially in $d$ smaller than $\Ep$ for all $p$. Moreover, since $n\in \Win,$ we also have 
    \begin{align}\label{compar2345}
    \left|\frac{n}{\Ep}-1\right|\le \exp{(-cd)}.
    \end{align}

    Further, again by Chebyshev's inequality at $p = \f 23$, 
    \[
        \left|\sum_v \varphi_v^2 -{\f 12}\right| \le \exp{(-cd)},
    \]
	as $\E[\sum \varphi_v^2] = 2^{d - 1} (1 - 3p/4)^d = {\f 12}$ and $\Var\Rnd{\sum_v \varphi_v^2} = \sum_v \Var(\varphi_v^2) \leq 2^{d - 1}
    (3/8)^d$. Thus, putting the above together we get that with probability $1-\exp{(-cd)},$
    
    $$|\t_d-\t_0|\le \exp{(-cd)}$$
where     \begin{align}
        \label{eq:theta1}
        \t_d &\coloneq \fr{1}{2} \cdot \fr{n(n - 1)}{\Phi^2} \cdot \sum \varphi_v^2, \text{ and }
         \t_0\coloneq{\f 14}.  
    \end{align}
    It is this ${\f 14}$ that appears in the statement of the lemma.

We now establish the bounds on $Y^{(1)}_d$, $Y^{(3)}_d$, $Y^{(3)}_d$ with explicit $p$ dependence which will turn out to
be handy in the forthcoming analysis of the $p>\f23$ case as well. 
For $Y^{(1)}_d$, observe that
\begin{align}
        \label{eq:nbrx}
        \Ex{\sum_{u \sim_2 v} \varphi_u \varphi_v} \lesssim d^2 2^d (1 - p/2)^{2d}
    \end{align}
which is exponentially small in $d$ as long as $p>2-\sqrt 2,$ and hence in particular for the entirety of the interval $[\f23,1).$

For $Y^{(2)}_d$, it suffices to observe that 
    \[
        \Ex{\sum_v \varphi_v^3} \lesssim 2^d (1 - 7p/8)^d \text{ which at }p=\f23 \text{ is }  2^d (5/12)^d.
    \]

    Finally, for $Y^{(3)}_d$, a straightforward algebra shows that   \[
        Y^{(3)}_d = \Fr{n}{\Phi}^4 \Fr{\sum \varphi_v^2}{\sqrt{n}}^2,
    \] and further note that since $n \geq (1 - o(1)) \Ep \gtrsim (1 - o(1)) (2 - p)^d$, at $p = \f 23$ we have
    \[
        \Ex{\fr{\sum_v \varphi_v^2}{\sqrt{n}}} \lesssim (2 - 3p/2)^d \cdot (1 - o(1)) \cdot (2 - p)^{-d/2}.
    \] 
Simple applications  of Markov's inequality, along with the exponentially decaying bounds on the above expectations
conclude the proof for $p = \f 23$.

    For $p > \f 23$, the argument is even simpler. By a straightforward  calculation
    \begin{align*}
        \E_\pi N_\Nbr +\E_\pi N_\Rep &\lesssim d^2n^2 \pi(\xi_1 = \xi_2) \\
                      &= d^2n^2 \fr{\sum \varphi_v^2}{(\sum \varphi_v)^2} 
                      = d^2\Fr{n}{\Phi}^2 \cdot \sum \varphi_v^2.
    \end{align*} 
	The first inequality follows by observing that $\pi(\xi_1 \sim \xi_2)=\frac{\sum_{u\sim v}\varphi_u\varphi_v}{\Phi^2}\le
	d^2 \frac{\sum_{u}\varphi^2_u}{\Phi^2}$ using the simple inequality $\varphi_u\varphi_v\le \frac{\varphi^2_u+\varphi^2_v}{2}.$ The
	proof is now complete on plugging in the earlier comparison between $n$ and $\Phi$ from \eqref{compar2344} and
	\eqref{compar2345}, and observing that $\E[\sum \varphi_v^2]$  is, up to constants, $2^d (1 - 3p/4)^d$, which is
	exponentially small when $p > \f 23$.

\end{proof}

\subsection{Proof of Lemma \ref{lem:iop_birthday234}}
\label{sec:proof_of_iop_birthday}

Here we combine the results of this section in order to prove Lemma \ref{lem:iop_birthday234}.
We start with a simple lemma that will be useful multiple times later.

\begin{lemma}\label{useful245}
Both $\fr{\exp(\Phi)}{\exp(\mu_p)}$ and its inverse are  bounded in probability. 
\end{lemma}
\begin{proof}
This is immediate since $\E(\Phi)=\mu_p$ and $\Var(\Phi)=O(1)$.
\end{proof}

Having recorded this we now proceed to using the results  of the previous subsection to show that for $m \in \Win$, the 
sum $\sum_{S \in \Good^\Even_m} 2^{-\N(S)}$ is well-approximated by $\fr{\Phi^m}{m!}$ times
the probability of non-collision when sampling $m$ points from the random probability measure
$\pi$.

\begin{lemma}
    \label{lem:goodisnoncollision}
    Let $\zeta_p = 1$ for $p > \fr{2}{3}$ and $\zeta_{\f 23} = e^{-{\f 14}}$.
    We have
    \begin{align*}
        \sum_{m \in \Win} \fr{\Abs{\sum_{S \in \Good^\Even_m} 2^{-\N(S)} - \zeta_p \fr{\Phi^m}{m!}}}
            {\sigma_p \exp(\mu_p)} \pto 0.
    \end{align*}
\end{lemma}

\begin{proof}
First recall that
\begin{align*}
    \sum_{S \in \Good^{\cE}_m} 2^{-\N(S)}
    = \sum_{S \in \Good^{\cE}_m} \prod_{v \in S} \varphi_v
    = \fr{\Phi^m}{m!} (1 - \pi(\Col_m)).
\end{align*}
So the expression in question is
\begin{align*}
    \fr{\sum_{m \in \Win} \fr{\Phi^m}{m!} \cdot \Abs{1 - \pi(\Col_m) - \zeta_p}}
    {\sigma_p \exp(\mu_p)}.
\end{align*}
Since $\pi(\Col_m)$ is increasing in $m$ (as the chance of a collision increases as more samples are drawn), this is in the interval
\begin{align*}
    \fr{\sum_{m \in \Win} \fr{\Phi^m}{m!}}{\exp(\mu_p)}
    \cdot \frac{\Box{\min(g_{m_-}, g_{m_+}), \max(g_{m_-}, g_{m_+})}}{\sigma_p},
\end{align*}
where $m_\pm$ are defined by $\Win = [m_-, m_+]$ and the gaps $g_{m} \df \Abs{1 - \pi(\Col_m) - \ze_p}$.
Now for $p\ge \f23,$ the prefactor, which is at most
$\fr{\exp(\Phi)}{\exp(\mu_p)}$, by Lemma \ref{useful245}, is bounded in probability.
Now let us treat the case $p=\f23$ first, where $\sigma_p=\frac{1}{\sqrt2}.$  In that case by Lemma \ref{lem:poisson},
\begin{align*}
    1 - \pi(\Col_{m_+}) - \zeta_p \pto 0,
    \quad \text{and} \qquad
    1 - \pi(\Col_{m_-}) - \zeta_p \pto 0,
\end{align*}
since $m_\pm \in \Win,$ which finishes the proof. In the case $p> \f23,$ $\zeta_p=1,$ and so we have to bound
$\frac{\pi(\Col_{m_+})}{\sigma_p}$. Note however that while $\sigma_{d}$ appearing in the denominator is exponentially
small, the numerator by Lemma \ref{lem:poisson} is bounded $\frac{d^2{m_+}^2}{\Phi^2}\sum_{v\in \cE}\varphi_v^2.$ Again the
pre-factor is bounded by, say, \eqref{compar2344} and \eqref{compar2345}, while $\E[\sum \varphi_v^2]$  is, up to
constants, $\sigma_p^2$.  Thus, a simple application of Markov's inequality finishes the proof.

\end{proof}

The proof of Lemma \ref{lem:iop_birthday234} is now immediate. 

\begin{proof}[Proof of Lemma \ref{lem:iop_birthday234}]
Again we set $\Phi = \Phi_\Even$; the same proof works for $\Phi_\Odd$.
We can write the expression in the lemma as a sum of three terms, all of
which are small by the results of this section.
Indeed,
\begin{align}
    \label{eq:birthday_full_disp}
    \fr{\sum_{S \in \Good^\Even} 2^{-\N(S)} - \zeta_p \exp(\Phi_\Even)}
        {\sigma_p \exp(\mu_p)}
    &= \fr{\sum_{m \notin \Win} \sum_{S \in \Good^\Even_m} 2^{-\N(S)}}{\sigma_p \exp(\mu_p)}
        - \zeta_p \fr{\sum_{m \notin \Win} \fr{\Phi^m}{m!}}{\sigma_p \exp(\mu_p)} \\
    \intertext{\hfill (both terms converge to $0$ in probability by Lemma \ref{lem:relevantalpha})} 
    &\quad + \sum_{m \in \Win} \fr{\sum_{S \in \Good^\Even} 2^{-\N(S)} \nonumber
        - \zeta_p \fr{\Phi^m}{m!}}
        {\sigma_p \exp(\mu_p)}
    \intertext{\hfill (converges to $0$ in probability by Lemma \ref{lem:goodisnoncollision})}
    \nonumber
\end{align}
This finishes the proof of the first part. The final conclusion is then immediate, since by Lemma \ref{useful245},
$\frac{\exp(\mu_p)}{\exp(\Phi_{\cE})}$ is bounded in probability and $\sigma_p=O(1).$ 
\end{proof}

\section{Geometry of a typical independent set}
\label{sec:geometry}
In this section, we prove Theorem \ref{thm:geometry} and en route, the key Proposition \ref{lem:sampling1}. We restate
the proposition, to include the $p=\f23$ case as well as a further Poisson convergence statement, that will allow us to
prove Theorem \ref{thm:geometry} quickly. We first state a more technical version of $\AS$  covering the $p=\f23$ case
that we will continue to refer to as $\AS$ henceforth. The only difference with Definition \ref{def:sampling} is the
first step of sampling $\cH$.
\begin{definition}[Approximate sampler]
    \label{def:sampling1}
    Given a realization of the percolation $\Qp$, we sample an independent set $S \sim \AS$ as follows:
\begin{enumerate}
	\item Choose a side $\cH = \Even$ or $\Odd$ with probabilities $\frac{e^{\Phi_\cE}}{e^{\Phi_\cE}+e^{\Phi_\cO}}$ and
		$\frac{e^{\Phi_\cO}}{e^{\Phi_\cE}+e^{\Phi_\cO}}$ respectively.
	\item For each vertex $v \in \cH$, pick it independently with probability $\fr{2^{-N_p(v)}}{1 + 2^{-N_p(v)}}$, and collect these
        vertices in the set $S_1$.
    \item For each vertex $u$ in the other side $\cH^\c$ not in $N_p(S_1)$, i.e., for every $u \in \cH^\c \setminus
        N_p(S_1)$, pick it independently with probability ${\f 12}$. Collect these vertices in the set $S_2$.
    \item Output $S = S_1 \cup S_2$.
\end{enumerate}
\end{definition}

\begin{proposition}[Properties of $\AS$]
    \label{lem:sampling2} For $p\ge \f23$,
    \[
        \TV{\AS - \US} \pto 0.
    \]
Further, \[\TV{|S_1|-\Pois(\mu_p)} \pto 0.\]
\end{proposition}

The plausibility of the Poisson convergence statement is apparent from  Lemma \ref{lem:goodisnoncollision} which
suggests that the number of independent sets with the size of the minority side, say the intersection with $\cE,$ being
$m$ is approximately $\frac{\Phi_{\cE}^m}{m!}.$

However, before proceeding with the proof we quickly remark that in the $p>\f23$ case, the outcomes of Definitions
\ref{def:sampling} and \ref{def:sampling1} are very close in total variation norm, with high probability, which then
will allow us to simply work with the latter throughout this section.  To see this note that since $\sigma_p$
exponentially decays in $d$ when $p>\f23,$ both $\Phi_{\cE}-\mu_p$ and $\Phi_{\cO}-\mu_{d}$ converge to zero in
probability, and hence the sampling probabilities for $\cH,$ i.e., $\frac{e^{\Phi_\cE}}{e^{\Phi_\cE}+e^{\Phi_\cO}}$ and
$\frac{e^{\Phi_\cO}}{e^{\Phi_\cE}+e^{\Phi_\cO}}$ converge to $\f12$ in probability.

\begin{proof}[Proof of  Proposition \ref{lem:sampling2}]
    Recall the definition of $\Nice$ sets from \eqref{structureinput2}: a set $S$ is in $\Nice$ if 
    \begin{itemize}
        \item either, $S \cap \Even \in \Good^\Even$ and $|S \cap \Odd| > 2^d/d^2$, 
        \item or, $S \cap \Odd \in \Good^\Odd$ and $|S \cap \Even| > 2^d/d^2$.
    \end{itemize}
	Let us also call the sets of the former kind $\NiceE$ and sets of the latter kind $\NiceO$. Note that these two sets
	are disjoint. Our
    first step is to show that $\AS$ is almost uniform when restricted to $\Nice$. 

    To that end, let $S \in \Nice$, and further, without loss of generality, suppose $S \in \NiceE$. Call $S_1 = S \cap
    \Even \in \Good^\Even$, and $S_2 = S \setminus S_1$. 
    Note that $\AS$ is a mixture of $\AS^\Even$ and $\AS^\Odd$ which are the distributions given $\cH=\cE$ and $\cH=\cO$ respectively.
  We first bound  $\AS^\Odd(\NiceE)$ and its symmetric counterpart. Note that this is bounded by the probability $\AS^\Odd (|S_2|\ge \fr{2^d}{d^2}).$
Now under $\AS^\Odd$, $\E(|S_2|)=\Phi_{\cO}.$
Thus, using this reasoning for both odd and even we get,    
\begin{equation}\label{bound57}
\AS^\Odd(\NiceE)+\AS^\Even(\NiceO) \le d^2\frac{\Phi_{\cO}+\Phi_{\cE}}{2^d}.
\end{equation}
  On the other hand, conveniently, $\AS^\cE$ restricted to $\NiceE$ is uniform (the corresponding statement holds for $\cO$).
To see this note that the probability that $S\in \NiceE$ is chosen 
    by $\AS^{\cE}$ is
	\begin{align}
		\label{eq:unifapprox}
		\prod_{v \notin S_1} \frac{1}{1 + 2^{-N_p(v)}} \cdot \prod_{v \in S_1} \frac{2^{-N_p(v)}}{1 + 2^{-N_p(v)}} 
			&\cdot \P(\text{uniform subset of $\cH^\c \setminus N_p(S_1)$ is $S_2$}) \\ 
			&= \frac{\prod_{v \in S_1} 2^{-N_p(v)} \cdot \Rnd{2^{2^{d - 1} - N_p(S_1)}}^{-1}}{\prod_{v \in \Even} (1 + 2^{-N_p(v)})} \nonumber \\
			&= \frac{1}{2^{2^{d - 1}} \prod_{v \in \Even} (1 + 2^{-N_p(v)})}, \nonumber
	\end{align}
    a quantity that does not depend on $S$ since $\sum_{v \in S_1} N_p(v) = N_p(S_1)$ as $S_1$ is well-separated by
    definition. 
On the other hand $\US$ is uniform on $\Nice$ by definition. 
Further by \eqref{structureinput2},
\begin{align*}
1-\frac{|\Nice|}{\Cnt}\pto 0.
\end{align*}
We next show that the above continues to hold for  $\AS^{\cE}$ and $\AS^{\cO}$ i.e.,
\begin{align}\label{conv56789}
1-\AS^{\cE}{(\Nice^{\cE})}\pto 0, \,\,\text{and,}\,\,  1-\AS^{\cO}{(\Nice^{\cO})}\pto 0,
\end{align} 
with even simpler arguments due to its explicit description.  We will just describe the first case. Observe that the
outcome of $\AS^{\cE}$, which we call $S$, fails to be in $\Nice^{\cE}$ for three
    possible reasons:
    \begin{itemize}
        \item $|S \cap \Even|> 2^d/d^2$,
        \item $|S \cap \Odd|< 2^d / d^2$,
        \item $S \cap \Even$ is not well-separated.
    \end{itemize}
    The first case is bounded by the reasoning in \eqref{bound57}.
	For the second case, we need a slightly sharper argument than in \eqref{bound57} and need to use that the expected
	size of $S_1$ is $\Phi$ and the variance is at most $\sum_v \varphi_v = \Phi$ too (since $\Var(\Ber(p)) \leq p$). Thus,
	by Chebyshev's inequality, $|S_1| \leq 2\Phi$ with probability $1 - \frac{1}{\Phi}$.Then, on this event, $S_2$ is a
	uniform subset of a set of size at least $2^{d - 1} - 2d\Phi\ge \frac{2^{d}}{4}$ (the last inequality holding with
	high probability in the behavior of $\Phi$), so
    $|S_2| \geq 2^d / d^2$ with probability $1 - \exp{(-cd)}$.

    Finally, to rule out case 3, 
    observe that the expected number of {\it invalid pairs}, i.e., chosen vertices in $\cE$ which are 2-neighbors, is
    \[
        \sum_{u \sim_2 v} \varphi_u \varphi_v, \quad u, v \in \cH,
    \]
    so a first moment bound and \eqref{eq:nbrx} showing that this goes to zero in probability, we see that
    $\AS(\Nice^\c) \pto 0$.
At this point it will be convenient to let $\US^{\Even}$ denote the distribution of $\US$ \emph{conditioned} to be in
$\Nice^{\cE}$ and similarly define $\US^{\Odd}.$

We will now rely on the following straightforward bound which states that for any two probability measures $\mu$ and
$\nu$ on a common space, and a set $A$, such that $\mu$ and $\nu$ conditioned on $A$ are uniform,
\begin{align}\label{TVest675}
    \TV{\mu-\nu}\lesssim \mu(A^\c) +\nu(A^\c).
    \end{align}
Using this and setting $A$ to be $\Nice^{\cE}$ (and symmetrically for the odd side), it follows that 
\begin{align}\label{tv456}
\TV{\AS^{\cE}-\US^{\cE}}&\lesssim 1-\AS^{\cE}(\Nice^{\cE}).\\
\nonumber
\TV{\AS^{\cO}-\US^{\cO}}&\lesssim 1-\AS^{\cO}(\Nice^{\cO}).
\end{align}
Now notice that both $\US$ and $\AS$ admit the following representations as mixtures of distributions.
\begin{align}\label{mixture123}
\AS&= \frac{e^{\Phi_\cE}}{e^{\Phi_\cE}+e^{\Phi_\cO}}\AS^{\cE}+\frac{e^{\Phi_\cE}}{e^{\Phi_\cE}+e^{\Phi_\cO}}\AS^{\cO},\\
\US&=\frac{|\Nice^\cE|}{\Cnt}\US^{\cE}+\frac{|\Nice^\cO|}{\Cnt}\US^{\cO}\\
\nonumber
&+\frac{|\Nice^\c|}{\Cnt}\US(\cdot\mid \Nice^\c).
\end{align}
{
	By Lemma \ref{lem:iop_polymer1}, Lemma \ref{lem:bigsmallprim} and  Lemma \ref{lem:iop_birthday234}, \eqref{imp546}, the
	discussion above it, and \eqref{structureinput2},
we have $$\frac{|\Nice^\cE|}{2^{2^{d-1}}\zeta_p \exp{\Phi_{\cE}}}\pto 1,$$ and similarly for odd. Putting things
together, we have the following convergence:
}
\begin{align*}
\left|\frac{|\Nice^\cE|}{\Cnt}- \frac{e^{\Phi_\cE}}{e^{\Phi_\cE}+e^{\Phi_\cO}}\right|+ \left|\frac{|\Nice^\cO|}{\Cnt}-
\frac{e^{\Phi_\cO}}{e^{\Phi_\cE}+e^{\Phi_\cO}}\right|+\frac{|\Nice^\c|}{\Cnt}\pto 0.
\end{align*}
Note that this implies that the mixture distributions in \eqref{mixture123} are close to each other in total variation
distance.  This along with \eqref{tv456} and triangle inequality for total variation distance finishes the proof of the
first part of the lemma.

To prove the second part of the lemma, observe that it suffices to establish
    \begin{align}
        \label{eq:evensize}
        \TV{\cQ^\Even - \Pois(\mu_p)} \pto 0
    \end{align}
    where $\cQ^\Even$ is the law of $|S \cap \Even|$ under $S \sim \AS^\Even$, since a
    similar result holds for the $\Odd$ side, and then the claim in the statement follows by triangle inequality. To
    simplify the goal further, note that by \eqref{conv56789}, we have
    \[
        \AS^\Even(S \cap \Even \notin \Good^\Even) \pto 0.
    \]
	Thus,
    \[
        \TV{\AS^\Even(\cdot\ |\ S \cap \Even \in \Good^\Even) - \AS^\Even} \pto 0.
    \]
    Call the first measure in the above expression $\ASE$, and for $S \sim \ASE$, let $\cP$ denote the law of $|S \cap
    \Even|$. We will show that $\TV{\cP - \Pois(\Phi_\Even)} \pto 0$. Recalling the definition of $\Good_m^\Even$ from the beginning 
	of Section \ref{sec:relevant}, we have via a similar argument as \eqref{eq:unifapprox}
    \[
		\cP(m) = \fr{\sum_{S_1 \in \Good^\Even_m} \frac{\prod_{v \in S_1}2^{-\N(v)}}{\prod_v (1 + 2^{-N_p(v)})}}{\sum_{S_1 \in \Good^\Even} 
			\frac{\prod_{v \in S_1}2^{-N_p(v)}}{\prod_v (1 + 2^{-N_p(v)})}} 
				= 
				\fr{\sum_{S_1 \in \Good^\Even_m} \prod_{v \in S_1}2^{-\N(v)}}{\sum_{S_1 \in \Good^\Even} \prod_{v \in S_1}2^{-N_p(v)}} 
				= 
				\fr{\sum_{S_1 \in \Good^\Even_m} 2^{-N_p(S_1)}}{\sum_{S_1 \in \Good^\Even} 2^{-N_p(S_1)}}
    \]
    and so, abbreviating $\Phi_\Even$ to $\Phi$,
    \[
        \Abs{\cP(m) - \P(\Pois(\Phi) = m)} = \Abs{\fr{\sum_{S_1 \in \Good^\Even_m} 2^{-\N(S_1)}}{\sum_{S_1 \in
        \Good^\Even} 2^{-\N(S_1)}} - \fr{\ze_p\fr{\Phi^m}{m!}}{\ze_p \exp(\Phi)}}
    \]
    introducing the $\ze_p$ in accordance with Lemma \ref{lem:goodisnoncollision} to prepare for the next step. For
    positive numbers $a, b, c, d$, 
    \begin{align}
        \label{eq:abcd}
        \Abs{\fr{a}{b} - \fr{c}{d}} = \Abs{\fr{a}{d} \cdot \Rnd{1 + \Box{\fr{d}{b} - 1}} - \fr{c}{d}} \leq \fr{\Abs{a - c}}{d} +
        \fr{a}{d} \Abs{\fr{d}{b} - 1},
    \end{align}
    so we have
    \begin{align}
        \label{eq:decomptv}
        \sum_{m \geq 0} \Abs{\cP(m) - \P(\Pois(\Phi) = m)} &\leq 
            \fr{1}{\ze_p \exp(\Phi)} \sum_{m} \Abs{\sum_{S_1 \in \Good^\Even_m} 2^{-\N(S_1)} - \ze_p \fr{\Phi^m}{m!}} \nonumber
            \\
                                                            &\phantom{=====} + 
                                                           \fr{\sum_{S_1 \in \Good^\Even}
                                                           2^{-\N(S_1)}}{\ze_p \exp(\Phi)} \cdot \Abs{\fr{\ze_p \exp(\Phi)}{\sum_{S_1 \in
                                                       \Good^\Even} 2^{-\N(S_1)}} - 1}
    \end{align}
Recall from \eqref{ratio987},
    \[
        \fr{\sum_{S_1 \in \Good^\Even} 2^{-\N(S_1)}}{\ze_p \exp(\Phi)} \pto 1,
    \]
    showing that the second term in \eqref{eq:decomptv} goes to zero in probability. It remains to prove the same for the
    first term. However, this essentially follows from Lemma \ref{lem:relevantalpha}
    and \ref{lem:goodisnoncollision}. To see how, first observe that by Lemma \ref{useful245} it suffices to control
    \[
        \fr{1}{\exp(\mu_p)} \sum_{m} \Abs{\sum_{S_1 \in \Good^\Even_m} 2^{-\N(S_1)} - \ze_p \fr{\Phi^m}{m!}}.
    \]
    The terms with $m \notin \Win$ go to zero in probability by Lemma \ref{lem:relevantalpha}. 
    The remaining terms, i.e., the sum over $m \in \Win$, goes to zero in probability by
    Lemma \ref{lem:goodisnoncollision}, concluding the proof of the fact that 
    \[
        \TV{\cP - \Pois(\Phi)} \pto 0.
    \]
    To finish the proof, it now suffices to show that 
    \[
        \TV{\Pois(\Phi) - \Pois(\mu_p)} \pto 0.
    \]
However, this is straightforward considering $\E(\Phi)=\mu_p$ and $\Var(\Phi)=O(1)$ and since Poisson variables with
large means behave like Gaussians, so as long as the difference in the means is much smaller than their fluctuations,
one should expect the Poisson distributions to be close. We record a precise bound in the appendix (see Lemma
\ref{lem:poissontv}), which then reduces things to asserting that $\sqrt{\Phi} - \sqrt{\mu_p} \pto 0$. This,
    however follows from the CLT for $\Phi$, i.e., Lemma \ref{lem:jointclt}, since
    \[
        \Abs{\sqrt{\Phi} - \sqrt{\mu_p}} = \fr{\Abs{\Phi - \mu_p}}{\sig_p} \cdot \fr{\sig_p}{\sqrt{\Phi} +
        \sqrt{\mu_p}} \leq \fr{\Abs{\Phi - \mu_p}}{\sig_p} \cdot \fr{\sig_p}{\sqrt{\mu_p}} \pto 0.
    \]
    The first factor is bounded in probability, and the second goes to zero, deterministically, as
    \[
        (2 - 3p/2)^d \ll \sqrt{(2 - p)^d}
    \]
    when $p \geq \f 23$.
\end{proof}

We finish off this section with a quick proof of Theorem \ref{thm:geometry}.

\begin{proof}[Proof of Theorem \ref{thm:geometry}]
	The proof is now an immediate consequence of Proposition \ref{lem:sampling2} since as proven, $\AS(\abs{S_1}< \abs{S_2})\pto
	1$ and thus by Proposition \ref{lem:sampling2} the defect side in $\US$ and $S_1$ in $\AS$ can be coupled with high
	probability with the law of the latter being close to a Poisson implying the same for the former.
    
\end{proof}

\def\pc{\fr{(1 + \lam)^2}{2\lam(2 + \lam)}}
 
\section{Extensions to the hard-core model}\label{hardcore3456}
As indicated in the introduction, in particular in \eqref{hard12} and \eqref{hard13}, all our main results extend beyond
the case of the uniformly sampled independent sets and their counts to the hard-core model defined in
\eqref{hardcoredef342}. While the arguments are essentially the same which we will not repeat, in this short section, we
discuss the counterparts of all the key inputs appearing throughout the paper. The arguments then work as is with these
alternate inputs plugged in.  As already stated, for a given $\lambda>0$, the putative transition location $\pc$ is less than $1$ only when $\lambda > \sqrt{2}-1$ which is the only regime which will be considered here (we will not be repeating this condition on $\lambda$ further).

The key difference stems from the effect of the defects. Note that the contribution to the partition function $\Cntl$
from independent sets solely on, say, the odd side is $(1+\lambda)^{2^{d-1}}.$ A single defect, $v$ on the even side,
changes this by $\lambda(1+\lambda)^{-N_{p}(v)}$ which, being the counterpart of $2^{-N_p(v)}$ will now be denoted by
$\varphi_v=\varphi_{v,\lam}.$ This leads to the corresponding definitions of $\Phi_{\cE}=\Phi_{\cE, \lam}$ and
$\Phi_{\cO}=\Phi_{\cO,\lam}.$

In analogy with the $\lam = 1$ case, the proofs proceed via  examining an approximation of the type 
\[
        \cZ_\lam = \cZ_{d, p, \lam} \approx (1 + \lam)^{2^{d - 1}} \Rnd{\exp(\Phi_{\Even, \lam}) + \exp(\Phi_{\Odd,\lam})},
\]
and potentially with a multiplicative correction
at the critical $p$, on account of the ``collisions'', as suggested by Section \ref{sec:birthday_full}.
Observe that for any $S \sse \Even$, the overall contribution to $\cZ_\lam$ of independent sets $I$ with $I \cap \Even = S$ is
\[
    \sum_{B \sse \Odd \setminus N_p(S)} \lam^{|S| + |B|} = \lam^{|S|} (1 + \lam)^{2^{d - 1} - N_p(S)}.
\]
Thus, if $S$ is well-separated, this contribution is simply
\[
    (1 + \lam)^{2^{d - 1}} \prod_{v \in S} \lam(1 + \lam)^{-N_p(v)}.
\]

Let us next present a few instructive moment computations mimicking Lemmas \ref{lem:xvmoments} and
\ref{lem:exvarcov}.

\begin{lemma}[Moments]
    \label{lem:momentslam}
    We have the following moments:
    \begin{align*}
        \E[\varphi_v^k] &= \lam^k\Rnd{1 - p + \fr{p}{(1 + \lam)^k}}^d, \quad k \geq 1, \\
        \Var(\varphi_v) &= \lam^2 \Rnd{1 - p + \fr{p}{(1 + \lam)^2}}^d \cdot (1 - o(1)), \\
        \E[\Phi_\Even]  &= \fr{\lam}{2}\Rnd{2\Box{1 - \fr{\lam p}{1 + \lam}}}^d, \\
        \Var(\Phi_\Even) &= \fr{\lam^2}{2} \Rnd{2\Box{1 - p + \fr{p}{(1 + \lam)^2}}}^d \cdot (1 - o(1)), 
    \end{align*}
    with similar results for $\Phi_\Odd$.
\end{lemma}
\begin{proof}
    The proof is straightforward. For each $k \geq 1$ we have
    \[
        \E[\varphi_v^k] = \lam^k\Ex{(1 + \lam)^{-k \cdot \Bin(d, p)}} = \lam^k \Rnd{\Ex{(1 + \lam)^{-k\cdot\Ber(p)}}}^d = \lam^k
        \Rnd{1 - p + \fr{p}{(1 + \lam)^k}}^d.
    \]
    Thus,
    \begin{align*}
        \Var(\varphi_v) = \E[\varphi_v^2] - \E[\varphi_v]^2 &= \lam^2 \Box{\Rnd{1 - p + \fr{p}{(1 + \lam)^2}}^d - \Rnd{1 - p + \fr{p}{1 +
        \lam}}^{2d}}  \\
                                          &= \lam^2 \Rnd{1 - p + \fr{p}{(1 + \lam)^2}}^d \cdot (1 - o(1))
    \end{align*}
    since
    \[
        \Rnd{1 - p + \fr{p}{1 + \lam}}^2 < 1 - p + \fr{p}{(1 + \lam)^2}, \quad p \in (0, 1), \lam > 0.
    \]
    The remaining estimates follow from the fact that $\Phi$ is a sum of $2^{d - 1}$ i.i.d. copies of $\varphi_v$.
\end{proof}

In particular, note that since 
\[
    2\Box{1 - p + \fr{p}{(1 + \lam)^2}} < 1 \iff p > \pc, 
\]
the variance of $\Phi$ is exponentially small when $p > \pc$ and constant at $p = \pc$, with the constant being
$\lam^2/2$. It might be reassuring to check that  indeed $\pc = 2/3$ when $\lam = 1$.\\

In the rest of the section we discuss the modifications needed to deliver the key inputs which serve as the main
ingredients in the proof of Theorem \ref{thm:main}.

The first is the CLT result Lemma \ref{lem:jointclt} and its straightforward consequences recorded in Lemma
\ref{lem:iop_jointclt}. The next is Lemma \ref{lem:relevantalpha} which shows that the terms that contribute most to
$\exp(\Phi)$, i.e., where $\Phi$ is either $\Phi_\Even$ or $\Phi_\Odd$, are those with index in a $\Win \df [\mu_p -
\mu_p^{0.6}, \mu_p + \mu_p^{0.6}]$ (the choice of $0.6$ is, again, arbitrary, and anything slightly larger than $0.5$
suffices). 
The final input is Lemma \ref{lem:poisson} which sharply estimates the collision probability by proving closeness of the
number to a Poisson variable.

In the general $\lambda$ case, 
the joint CLT result for $(\Phi_\Even, \Phi_\Odd)$ and its consequence as in Lemma \ref{lem:jointclt} hold verbatim,
once appropriate replacements of the scaling and centering constants have been made. To this end recall the constants
from \eqref{modconstant123},
\begin{align*}
    \mu_{p, \lam} &= \fr{\lam}{2} \Rnd{2 - \fr{2\lam p}{1 + \lam}}^d = \E \Phi_\Even\\
    \sig_{p, \lam}^2 &= \fr{\lam^2}{2}\Rnd{2\Box{1 - p + \fr{p}{(1 + \lam)^2}}}^d = \Var(\Phi_\Even), \quad \text{up to
    $1 - o(1)$ correction}.
\end{align*}
{\it For the remainder of this section, we will abbreviate $\mu_{d, \lam}$ to $\mu_p$ and $\sig_{d, \lam}$ to $\sig_p$,
and keep the $\lam$-dependence implicit.}

\newcommand{\pl}{{p, \lam}}

\begin{lemma}
    \label{lem:jointcltlam}
    In the notation above, the following joint CLT holds for $\Phi_\Even$ and $\Phi_\Odd$ for all $p$ and $\lam$ fixed:
\begin{align*}
    \Rnd{
		\fr{\Phi_\Even - \mu_{p, \lam}}{\sigma_{p, \lam}}, \fr{\Phi_\Odd - \mu_{p, \lam}}{\sigma_{p, \lam}}
    } \weakto \Nor{0}{1} \otimes \Nor{0}{1}.
\end{align*} 
    Consequently, for $p > \pc$ defined above, we have 
    \[
		\fr{\exp(\Phi_\Even) + \exp(\Phi_\Odd) - 2\exp(\mu_\pl)}{\sqrt{2}\sig_\pl \exp(\mu_\pl)} \weakto \Nor{0}{1},
    \]
    and at $p = \pc$, 
    \[
        \fr{\exp(\Phi_\Even) + \exp(\Phi_\Odd)}{\exp(\mu_\pl)} \weakto \Exp{\fr{\lam W_1}{\sqrt{2}}} + \Exp{\fr{\lam W_2}{\sqrt{2}}}
    \]
    where $W_1, W_2 \sim \Nor{0}{1}$, i.i.d.
\end{lemma}
\begin{proof}
    For the joint CLT, the same proof as that of  Lemma \ref{lem:jointclt} applies. For the
    distributional limits of exponentials, the only modification from Lemma \ref{lem:iop_jointclt} is that
	$\exp(W_1/\sqrt{2})$ is replaced by $\exp(\lam W_1 / \sqrt{2})$ (and similarly for $W_2$). This is due to the fact
	that at $\pc$, $\sig_\pl^2$ 
	is now $\fr{\lam^2}{2}$ and not simply ${\f 12}$.
\end{proof}

The same argument as in  Lemma \ref{lem:relevantalpha} implies that when $p \geq \pc$, the terms that contribute most to
$\exp(\Phi)$, are those with index in $\Win$ and hence this does not warrant any further commentary.

We end this section with a discussion on the counterpart of Lemma \ref{lem:poisson}. This is the
content of the next lemma. We will use the same notations  $\Col_n, \Rep_n$ as well as $N_\Col, N_\Rep, N_\Nbr$ as discussed 
right before Lemma \ref{lem:poisson}, but now for
samples $\xi_1, \ldots, \xi_n \sim \pi.$ Here $\pi(v)$ is  the random measure with weights proportional to $\varphi_{v}
$ depending on $\lambda,$ as defined above. Given this, the following statement replaces Lemma \ref{lem:poisson} where,
as before, the law of $N_\Col$ is referred to as
$\cL_{n, d}$ (with now an implicit dependence on $\lam$ as well).

\begin{lemma}
    \label{lem:poissonlam}
    The law $\cL_{n, d}$ satisfies that for $p = \pc$, 
    \[
		\TV{\cL_{n, d} - \Pois\Fr{\lam^2}{4}} \pto 0
    \]
    and for $p > \pc$,
    \[
        \TV{\cL_{n, d} - \del_0} \pto 0,
    \]
    as $d \to \infty$ as long as $n = \mu_\pl \cdot (1 \pm o(1))$.

     As a consequence, for $p = \pc$, 
    \[
		\pi(\Col_n) \pto 1 - e^{-\fr{\lam^2}{4}},
    \]
    and when $p > \pc$,
    \[
        \pi(\Col_n) \pto 0.
    \]
	In the latter case, as in Lemma \ref{lem:poisson}, an explicit bound is also given by
	$\frac{d^2n^2}{\Phi^2}\sum_{v\in \cE}\varphi_v^2.$
\end{lemma}
\begin{proof}
    The proof of Lemma \ref{lem:poisson} essentially goes through verbatim. The only computation we wish to
	demonstrate is the origin of the parameter $\fr{\lam^2}{4}$ for $p = \pc$.
    For this, we recompute equation \eqref{eq:theta1}, i.e., we
    determine the (in probability) limit of 
    \[
        \t_d \df \binom{n}{2}\fr{\sum_v \varphi_v^2}{\Rnd{\sum_v \varphi_v}^2} \quad \text{(see \eqref{eq:thetadef})}
    \]
    where $n/\mu_\pl \to 1$. Via the CLT for $\Phi$, $(\sum_v \varphi_v)^{-2} \binom{n}{2} \pto \fr{1}{2}$. Also,
    \[
        \Ex{\sum_v \varphi_v^2} = \sig_\pl^2 = \fr{\lam^2}{2}, \quad \text{at $p = \pc$.}
    \]
    This is also the limit in probability of $\sum_v \varphi_v^2$ since
    \[
        \Var\Rnd{\sum_v \varphi_v^2} \leq \sum_v \E[\varphi_v^4] \lesssim \lam^4\Rnd{2\Box{1 - p + \fr{p}{(1 + \lam)^4}}}^d =
        \lam^4 \Fr{\lam^2 + 2\lam}{\lam^2 + 2 \lam + 1}^d \ll 1,
    \]
    when $p = \pc$, as may be verified by a direct calculation. Putting everything together,
 $\t_d \pto \fr{\lam^2}{4}$.
\end{proof}

We end this article by recording the elementary proofs of some of the lemmas featuring in the main body.

\section{Appendix}\label{appendix234}
\begin{lemma}
\label{lem:poissontails}
Let $X \sim \mathrm{Poisson}(\lambda)$.
Then we have
\begin{itemize}
    \item (Upper Tail Bound). For any $t > 0$,
    \begin{align*}
        \P[X > \lambda + t \sqrt{\lambda}]
        \leq \exp\Rnd{- \fr{t^2}{2} + \fr{t^3}{2 \sqrt{\lambda}}}.
    \end{align*}
    \item (Lower Tail Bound). For any $\eta > 0$ there is some constant $c_\eta$ such that 
    if $t \leq c_\eta \sqrt{\lambda}$, we have
    \begin{align*}
        \P[X < \lambda - t \sqrt{\lambda}]
        \leq \exp\Rnd{- \Rnd{\fr{1}{2} - \eta} t^2
        - \Rnd{\fr{1}{2} + \eta} \fr{t^3}{\sqrt{\lambda}}}.
    \end{align*}
    \item (Two-Sided Bound). In particular, if
    $t \leq \min\left\{c_{\fr{1}{6}}, \fr{1}{3}\right\} \cdot \sqrt{\lambda}$,
    we have
    \begin{align*}
        \P[|X - \lambda| > t \sqrt{\lambda}] \leq 2 \exp\Rnd{- \fr{t^2}{3}}.
    \end{align*}
\end{itemize}
\end{lemma}

\begin{proof}[Proof of Lemma \ref{lem:poissontails}]
Let $T_\pm = \lambda \pm t \sqrt{\lambda}$.
We will apply Chernoff's method.
First, for any $a > 0$,
\begin{align*}
    \P[X > T_+] \leq \fr{\exp(\lambda (e^a - 1))}{\exp(T_+ a)}.
\end{align*}
Let us choose $a$ such that $e^a = \fr{T_+}{\lambda}$ (such an $a$ is positive
since $T_+ > \lambda$).
This gives
\begin{align*}
    \P[X > T_+] &\leq \fr{\exp(T_+ - \lambda)}{\Fr{T_+}{\lambda}^{T_+}} = \fr{\exp(t \sqrt{\lambda})}
        {\Rnd{1 + \fr{t}{\sqrt{\lambda}}}^{\lambda + t \sqrt{\lambda}}} \\
    &\leq \exp \Rnd{ t \sqrt{\lambda} - 
        \Rnd{\fr{t}{\sqrt{\lambda}} - \fr{t^2}{2\lambda}}
        \Rnd{\lambda + t \sqrt{\lambda}} } = \exp \Rnd{- \fr{t^2}{2} + \fr{t^3}{2 \sqrt{\lambda}}},
\end{align*}
where in the second inequality we used the fact that
$1 + x \geq \exp\Rnd{x - \fr{x^2}{2}}$, which holds for all $x > 0$.
Similarly, for any $a > 0$ we have
\begin{align*}
    \P[X < T_-] = \P[- X > - T_-] \leq \fr{\exp(\lambda (e^{-a} - 1))}{\exp(- T_- a)}.
\end{align*}
Choosing $a$ such that $e^{-a} = \fr{T_-}{\lambda}$, we obtain
\begin{align*}
    \P[X < T_-] &\leq \fr{\exp(T_- - \lambda)}{\Fr{T_-}{\lambda}^{T_-}} = \fr{\exp(- t \sqrt{\lambda})}
        {\Rnd{1 - \fr{t}{\sqrt{\lambda}}}^{\lambda - t \sqrt{\lambda}}} \\
    &\leq \exp \Rnd{ - t \sqrt{\lambda} - 
        \Rnd{- \fr{t}{\sqrt{\lambda}} - \Rnd{\fr{1}{2} + \eta} \fr{t^2}{\lambda}}
        \Rnd{\lambda - t \sqrt{\lambda}}} \\
    &= \exp\Rnd{
        - \Rnd{\fr{1}{2} - \eta} t^2 - \Rnd{\fr{1}{2} + \eta} \fr{t^3}{\sqrt{\lambda}}
    }
\end{align*}
where in the second inequality we used the fact that
$1 - x \geq \exp\Rnd{- x - (\fr{1}{2}+\eta) x^2}$, which holds for all positive 
$x \leq \fr{4 \eta}{1 + 6 \eta + 8 \eta^2}$, which we will take to be $c_\eta$.
The two-sided bound follows from a combination of the two previous bounds, with 
a union bound.
\end{proof}

\begin{lemma}
    \label{lem:poissontv}
    The total-variation distance between $\Pois(\lam)$ and $\Pois(\mu)$ is, up to universal constants,
    at most $\Abs{\sqrt{\lam} - \sqrt{\mu}}$.
\end{lemma}
\begin{proof}
    See \cite{roos}*{Theorem 1(b)}.
\end{proof}

\bibliographystyle{alpha}
\bibliography{references}

\begin{bibdiv}
\begin{biblist}

\bib{stein}{incollection}{
      author={Baldi, Pierre},
      author={Rinott, Yosef},
      author={Stein, Charles},
       title={A normal approximation for the number of local maxima of a random function on a graph},
        date={1989},
   booktitle={Probability, statistics, and mathematics},
   publisher={Elsevier},
       pages={59\ndash 81},
}

\bib{balogh2015independent}{article}{
      author={Balogh, J{\'o}zsef},
      author={Morris, Robert},
      author={Samotij, Wojciech},
       title={Independent sets in hypergraphs},
        date={2015},
     journal={Journal of the American Mathematical Society},
      volume={28},
      number={3},
       pages={669\ndash 709},
}

\bib{b13}{inproceedings}{
      author={Barbier, Jean},
      author={Krzakala, Florent},
      author={Zdeborov{\'a}, Lenka},
      author={Zhang, Pan},
       title={The hard-core model on random graphs revisited},
organization={IOP Publishing},
        date={2013},
   booktitle={Journal of physics: Conference series},
      volume={473},
       pages={012021},
}

\bib{bgw2}{article}{
      author={Basu Roy~Chowdhury, Mriganka},
      author={Ganguly, Shirshendu},
      author={Winstein, Vilas},
       title={Independent sets in a percolated hypercube: the non-concentration regime},
        date={forthcoming},
}

\bib{chatterjee}{article}{
      author={Chatterjee, Sourav},
      author={Diaconis, Persi},
      author={Meckes, Elizabeth},
       title={Exchangeable pairs and {P}oisson approximation},
        date={2005},
}

\bib{c15}{article}{
      author={Coja-Oghlan, Amin},
      author={Efthymiou, Charilaos},
       title={On independent sets in random graphs},
        date={2015},
     journal={Random Structures \& Algorithms},
      volume={47},
      number={3},
       pages={436\ndash 486},
}

\bib{dingslysunregular}{article}{
      author={Ding, Jian},
      author={Sly, Allan},
      author={Sun, Nike},
       title={Maximum independent sets on random regular graphs},
        date={2016},
}

\bib{velenik}{book}{
      author={Friedli, Sacha},
      author={Velenik, Yvan},
       title={Statistical mechanics of lattice systems: a concrete mathematical introduction},
   publisher={Cambridge University Press},
        date={2017},
}

\bib{g11}{article}{
      author={Galvin, David},
       title={A threshold phenomenon for random independent sets in the discrete hypercube},
        date={2011},
     journal={Combinatorics, probability and computing},
      volume={20},
      number={1},
       pages={27\ndash 51},
}

\bib{galvinind}{article}{
      author={Galvin, David},
       title={Independent sets in the discrete hypercube},
        date={2019},
     journal={arXiv preprint arXiv:1901.01991},
}

\bib{g04}{article}{
      author={Galvin, David},
      author={Kahn, Jeff},
       title={On phase transition in the hard-core model on $\mathbb{Z}^d$},
        date={2004},
     journal={Combinatorics, Probability and Computing},
      volume={13},
      number={2},
       pages={137\ndash 164},
}

\bib{perkins}{article}{
      author={Jenssen, Matthew},
      author={Perkins, Will},
       title={Independent sets in the hypercube revisited},
        date={2020},
     journal={Journal of the London Mathematical Society},
      volume={102},
      number={2},
       pages={645\ndash 669},
}

\bib{j22}{article}{
      author={Jenssen, Matthew},
      author={Perkins, Will},
      author={Potukuchi, Aditya},
       title={Independent sets of a given size and structure in the hypercube},
        date={2022},
     journal={Combinatorics, Probability and Computing},
      volume={31},
      number={4},
       pages={702\ndash 720},
}

\bib{kahn}{article}{
      author={Kahn, Jeff},
       title={An entropy approach to the hard-core model on bipartite graphs},
        date={2001},
     journal={Combinatorics, Probability and Computing},
      volume={10},
      number={3},
       pages={219\ndash 237},
}

\bib{sapozhenko}{article}{
      author={Korshunov, Aleksej~D},
      author={Sapozhenko, Alexander~A},
       title={The number of binary codes with distance 2},
        date={1983},
     journal={Problemy Kibernet},
      volume={40},
       pages={111\ndash 130},
}

\bib{ks}{article}{
      author={Kronenberg, Gal},
      author={Spinka, Yinon},
       title={Independent sets in random subgraphs of the hypercube},
        date={2022},
     journal={arXiv preprint arXiv:2201.06127},
}

\bib{panchenko2013sherrington}{book}{
      author={Panchenko, Dmitry},
       title={The sherrington-kirkpatrick model},
   publisher={Springer Science \& Business Media},
        date={2013},
}

\bib{roos}{article}{
      author={Roos, Bero},
       title={Improvements in the {P}oisson approximation of mixed {P}oisson distributions},
        date={2003},
     journal={Journal of Statistical Planning and Inference},
      volume={113},
      number={2},
       pages={467\ndash 483},
}

\bib{ross}{article}{
      author={Ross, Nathan},
       title={Fundamentals of {S}tein’s method},
        date={2011},
     journal={Probability Surveys},
      volume={8},
       pages={210\ndash 293},
}

\bib{hypergraph_containers}{article}{
      author={Saxton, David},
      author={Thomason, Andrew},
       title={Hypergraph containers},
        date={2015},
     journal={Inventiones mathematicae},
      volume={201},
      number={3},
       pages={925\ndash 992},
}

\end{biblist}
\end{bibdiv}

\end{document}